\documentclass[sn-mathphys,Numbered]{sn-jnl}

\usepackage{enumitem}
\usepackage{graphicx}
\usepackage{multirow}
\usepackage{amsmath,amssymb,amsfonts}
\usepackage{amsthm}
\usepackage[dvipsnames]{xcolor}
\usepackage{textcomp}
\usepackage{manyfoot}
\usepackage{booktabs}

\newtheorem{theorem}{Theorem}[section]
\newtheorem{proposition}[theorem]{Proposition}
\newtheorem{definition}[theorem]{Definition}

\newtheorem{remark}[theorem]{Remark}
\newtheorem{corollary}[theorem]{Corollary}
\newtheorem{Lemma}[theorem]{Lemma}
\newtheorem{assumption}[theorem]{Assumption}

\newcommand{\R}{\mathbb{R}}

\newcommand{\C}{\mathbb{C}}

\graphicspath{{figures/}}

\hypersetup{hidelinks}

\begin{document}

\title[Non static exponential turnpike with boundary conditions]{Non Static Exponential Turnpike Property for Optimal Control Problems with Symmetries and Boundary Conditions}

\author[1]{\fnm{Sofya} \sur{Maslovskaya}}\email{sofyam@math.upb.de}
\equalcont{These authors contributed equally to this work.}

\author[1]{\fnm{Sina} \sur{Ober-Bl\"obaum}}\email{sinaober@math.upb.de}
\equalcont{These authors contributed equally to this work.}

\author*[1]{\fnm{Boris} \sur{Wembe}}\email{wboris@math.upb.de}
\equalcont{These authors contributed equally to this work. This work has been supported by Deutsche Forschungsgemeinschaft (DFG), Grant No. OB 368/6-1, AOBJ: 716309, and Grant No. MA 12013/1-1, AOBJ: 716308.}

\affil[1]{\orgdiv{Department of Mathematics}, \orgname{Paderborn University}, \orgaddress{\country{Germany}}}

\abstract{
Optimal control problems with symmetries often admit a non stationary turnpike property, called trim turnpike, in which optimal solutions converge to symmetry-induced trajectories called trims. We establish an exponential trim turnpike property for problems with cyclic variables, as arising from Abelian symmetries, when the final values of all state variables are prescribed. In this setting, optimal solutions converge to a trim that itself depends on the time horizon and the boundary conditions. The key contributions of this work are a characterization of the trim turnpike trajectory, obtained through an adapted formulation of the reduced problem and a midpoint anchoring, and the derivation of exponential estimates with constants independent of the time horizon. The results are illustrated on a linear--quadratic problem, a nonlinear problem with two cyclic variables, and the Kepler orbital transfer with prescribed final phase.}



\maketitle

\section{Introduction}
\label{sec:intro}

The \emph{turnpike phenomenon} is one of the central organizing principles in long-horizon optimal control. Roughly speaking, it asserts that optimal trajectories spend most of the time horizon close to a distinguished reference object, while the influence of the initial and terminal conditions is essentially confined to boundary layers near the endpoints.
In the classical setting, this reference object is a steady state solving an associated static optimization problem, and the corresponding \emph{exponential turnpike property}
states that the convergence toward this steady state occurs at an exponential rate. This phenomenon has been identified in various applications of optimal control \cite{Caillau2022, Cots2021, Djema2021, McKenzie1986} and extensively studied both from the viewpoint of dynamical systems and from that of optimization and PDE analysis; see, among many others, \cite{Porretta2013, TreZua2015}.

In its most familiar form, the turnpike property is associated with static optimal pairs. However, many systems of practical interest exhibit a different behavior and their optimal solutions do not converge to steady points. This is notably the case for periodic systems or systems possessing symmetries where the relevant long-time behavior is better described by nonstationary reference motions. This observation has motivated the development of several \emph{non-static} versions of the turnpike phenomenon. More generally, manifold turnpikes and trim turnpikes capture situations in which the turnpike set is a sub-manifold or a symmetry-induced trajectory rather than a single equilibrium \cite{Faulwasser2022a, Flasskamp2025}. These extensions are especially relevant for mechanical systems, including orbital transfer models and general motion planning problems, which naturally admit a symmetry with respect to phase, position, or orientation.

More precisely, we say that an optimal control problem is symmetric with respect to a Lie group action if the running cost is invariant under the action and the control system is equivariant with respect to it. In case of mechanical systems, the equivariance of the control system stems from the variational principle behind, which involves a mechanical Lagrangian invariant with respect to the symmetry action. When the Lie group is Abelian, it is well known that there exist local coordinates, such that a part of state variables does not appear in the cost and the dynamics explicitly \cite{Grizzle1985}. In classical mechanics these coordinates are commonly referred to as \emph{shape} and \emph{cyclic} coordinates \cite{marsden2013}, where cyclic variables do not appear explicitly in the mechanical Lagrangian. Generalizing the concept to the optimal control setting, we call those state variables cyclic, which do not appear explicitly in the state dynamics nor in the cost.

In the presence of cyclic variables, the usual static turnpike picture is no longer sufficient to describe optimal solutions. A first important result in this direction is the \emph{linear turnpike theorem}, where the non-cyclic variables converge toward a steady regime while the cyclic variables evolve approximately linearly in time \cite{Tre:23}. On the other hand, in symmetric optimal control problems, recent work has shown that optimal
trajectories may converge exponentially toward a \emph{trim}, that is, a nonstationary reference trajectory generated by symmetry reduction \cite{Flasskamp2025}. Nevertheless, the latter analysis relies on a crucial simplifying assumption: the final conditions on the state variable are invariant with respect to the symmetry group action. In the case of problem formulations involving cyclic variables, this corresponds to free final conditions on the cyclic variables. Once endpoint conditions are imposed on the cyclic variables or, more generally, the final conditions are not invariant with respect to the symmetry, the classical symmetry reduction \cite{Ohsawa2013} does not allow to characterize the optimal solutions through the reduced optimal control problem on the principal bundle anymore, because the full end-point conditions must be taken into account. As a result, the construction of the appropriate turnpike reference becomes substantially more delicate.

This issue is not only of theoretical interest. In many applications, the terminal value of the cyclic variable is part of the problem data and cannot be ignored. Typical examples include transfer problems with prescribed final phase in space trajectory optimization, reorientation or rendezvous tasks with fixed terminal attitude, and mechanical systems where symmetry variables encode physically meaningful endpoint constraints; see for instance \cite{Betts2010}. In such situations, one expects the optimal trajectory to remain close to a symmetry-induced motion in the major middle part of the time interval, while being compatible with two-sided boundary conditions. The central question that must be addressed in order to characterize the turnpike in symmetric problems with general boundary conditions is how to define the correct reduced problem and the associated trim.

The purpose of this paper is to investigate this question for a class of finite-horizon optimal control problems with cyclic variables and fixed initial and final conditions, stated precisely in \eqref{eq:OCP} in Section~\ref{sec:setting}.

\paragraph{Relation to linear--quadratic turnpike problems.}
When $f_1$ and $g_1$ are linear, $F_2$ and $G_2$ are constant and $f^0$ is quadratic, \eqref{eq:OCP} is a linear--quadratic problem with cyclic variables, whose optimality system is a linear Hamiltonian boundary value problem. Exponential turnpike results are well understood in this framework, where they follow from the hyperbolicity of the Hamiltonian matrix under stabilizability and detectability assumptions \cite{TreZua2015, Porretta2013}. Since the cyclic variables do not appear in the cost, such problems are not detectable with respect to them; turnpike results without stabilizability or detectability, including the influence of terminal conditions, are obtained in \cite{Dario2021}. In the nonlinear setting with cyclic variables and prescribed final values, the linear turnpike theorem \cite{Tre:23} is the closest result to ours; we compare with it in detail in Section~\ref{sec:turnpike}.

\paragraph{Approach.}

Rather than postulating a reduced static problem as in \cite{Flasskamp2025}, we derive the reduced problem from the full Pontryagin boundary value problem. Since the adjoint of the cyclic variables is constant along extremals, it enters the Hamiltonian as a parameter, the \emph{cyclic multiplier}, and yields a family of reduced optimal control problems whose extremals coincide with the solutions of the reduced Hamiltonian boundary value problem; the relevant value of the multiplier is selected by the boundary conditions on the cyclic variables. This reduction is classical: the constancy of the cyclic adjoint is the conserved quantity given by Noether's theorem for optimal control \cite{Torres:04}, and the reduced problem is the optimal control analogue of Routh reduction in mechanics \cite{MarsdenRatiuScheurle2000}. We do not claim novelty for this reduction; our contribution is the turnpike analysis that it makes possible, uniformly with respect to the horizon.

Under local assumptions on a reference trim (positivity of a reduced Hessian and a Kalman condition for the full linearization), we prove that, for boundary data close to the reference trim, there is a locally optimal solution which stays close to it on the whole horizon, and that this solution enjoys an \emph{exponential trim turnpike property}: the shape variables, their adjoint and the control converge exponentially to the trim selected by the actual cyclic multiplier, and the cyclic variables to a trim trajectory anchored at the midpoint of the horizon. When the final condition on the cyclic variables is free, the cyclic multiplier vanishes and we recover the trim turnpike of \cite{Flasskamp2025}. Since the trim and its anchor are explicit, they can also be used to initialize shooting and homotopy methods on long horizons \cite{Cots2021}.

The rest of the paper is organized as follows. In Section~\ref{sec:setting} we state the problem, introduce trim primitives and collect the standing assumptions. In Section~\ref{sec:PMP_RBVP_ROCP} we derive the Pontryagin optimality system, introduce the reduced boundary value problem, and construct the corresponding reduced and static optimal control problems. Then in Section~\ref{sec:turnpike} we state and prove the exponential trim turnpike theorem. In Section~\ref{sec:applications} we illustrate the theory on linear and nonlinear problems with quadratic cost, and on the Kepler orbital transfer problem. Finally a conclusion is given in Section~\ref{sec:conclusion}, and the Appendix collects some technical proofs.

\section{Problem setting and standing assumptions}
\label{sec:setting}

\subsection{Problem statement}

We consider finite-horizon optimal control problems of the form
\begin{equation}
\begin{aligned}
\min_{u(\cdot) \in L^\infty([0,T])}\; J_T(u,x)
&= \int_0^T \left( f^0(x(t)) + \tfrac12\, u(t)^\top R\, u(t) \right)\,dt, \\
\text{s.t. }\quad
\dot x(t) &= f_1(x(t)) + F_2(x(t))\,u(t), \\
\dot y(t) &= g_1(x(t)) + G_2(x(t))\, u(t), \\
x(0)&=x_0,\quad x(T)=x_f, \\
y(0)&=y_0,\quad y(T)=y_f,
\end{aligned}
\label{eq:OCP}
\end{equation}
where $x(t)\in\mathbb{R}^n$ denotes the non-cyclic state (or shape variable), $y(t)\in\mathbb{R}^p$ is a cyclic variable, and $u(t)\in\mathbb{R}^m$ is the control input. The matrix {$R\in\mathbb{R}^{m\times m}$} is symmetric positive definite, and the matrices $F_2(x)\in\mathbb{R}^{n\times m}$ and $G_2(x)\in\mathbb{R}^{p\times m}$ collect the control vector fields. Taking $R$ constant is no loss of generality: if $R(x)=M(x)^\top M(x)$, the feedback $v=M(x)u$ reduces the cost to $\frac12|v|^2$ and preserves the form \eqref{eq:OCP}.
Throughout the paper, $x_f$ and $y_f$ denote the prescribed final values, while a subscript $T$, as in $(x_T(\cdot),y_T(\cdot),u_T(\cdot))$, refers to a solution on the horizon $[0,T]$.
The precise regularity assumptions are collected in Assumption~\ref{ass:standing} below.
The variable $y$ is cyclic in the sense that it does not appear explicitly in the dynamics nor in the cost functional. Nevertheless, the presence of boundary conditions on $y$ couples the evolution of the cyclic and the shape variables through the optimality conditions and plays a central role in the analysis.
This setting contains, as particular cases, linear--quadratic cyclic problems of the type considered in \cite{Dario2021}, as well as symmetric problems with Abelian symmetry group represented in shape-cyclic local coordinates \cite{Flasskamp2025}.

\subsection{Trim primitives}
\label{subsec:trims}

A particular role is played by trajectories that are stationary in the shape variables but not in the cyclic variables. If $f_1(x)+F_2(x)u=0$ for some constant pair $(x,u)$, then $t\mapsto(x,\,y_0+t\,(g_1(x)+G_2(x)u),\,u)$ is a trajectory of \eqref{eq:OCP} which is constant in $x$ while the cyclic variable moves with a constant, in general nonzero, velocity. Such trajectories are known as \emph{trim primitives}, or simply \emph{trims} \cite{Frazz2005}. In classical mechanics they correspond to relative equilibria, i.e.\ equilibria of the dynamics reduced by the symmetry \cite{marsden2013}. Geometrically, when the action is free and proper, shape-cyclic coordinates are adapted to a local trivialization of the induced principal bundle, the shape variables parameterizing the base and the cyclic ones the (Abelian) group: trims are stationary in the base and move only along the group direction.

\subsection{Standing assumptions}
\label{subsec:assumptions}

We collect here the assumptions used throughout the paper, before any result is stated. Assumptions specific to the turnpike results, which bear on a limit static point, are stated in Section~\ref{sec:turnpike} (Assumption~\ref{ass:trim}). Assumptions (S2)--(S3) are only used in the general discussion of Section~\ref{sec:lambda_selection}; the results of Section~\ref{sec:turnpike} only require (S1).

\begin{assumption}[Standing assumptions]
\label{ass:standing}

The following conditions hold.
\begin{itemize}
\item[(S1)] \emph{Regularity.} The functions $f^0:\R^n\to\R$, $f_1:\R^n\to\R^n$, $g_1:\R^n\to\R^p$ and the matrix fields $F_2:\R^n\to\R^{n\times m}$, $G_2:\R^n\to\R^{p\times m}$ are of class $C^3$, and $R\in\R^{m\times m}$ is symmetric positive definite.
\item[(S2)] \emph{Existence.} For every $T$ large enough, problem \eqref{eq:OCP} admits an optimal solution $(x_T(\cdot),y_T(\cdot),u_T(\cdot))$.
\item[(S3)] \emph{Normality.} The optimal solutions considered below admit a normal extremal lift, i.e.\ one with $p^0\neq 0$ in the Pontryagin Minimum Principle (Section~\ref{sec:PMP_RBVP_ROCP}).
\end{itemize}

\end{assumption}

\begin{remark}[On (S2)--(S3)]
\label{rem:existence}
\label{rem:normal}

If the controls are restricted to a compact set $\Omega\subset\R^m$, (S2) follows from the Filippov-type theorem \cite[Theorem~2.9]{trelat2024control}, provided that $(x_f,y_f)$ is reachable from $(x_0,y_0)$ in time $T$ and that the admissible trajectories do not blow up; the required convexity holds since the dynamics are affine in $u$ and the cost is convex in $u$. Choosing $\Omega$ as a large closed ball does not affect the local analysis of Section~\ref{sec:turnpike}, where the controls stay close to $\bar u$. For unbounded controls, (S2) is kept as an assumption; Proposition~\ref{prop:existence} does not rely on it. Restricting to normal extremals, as in (S3), is standard in turnpike analysis. For the solution of Proposition~\ref{prop:existence}, normality is part of the conclusion, and by Proposition~\ref{prop:existence}(ii) this solution has no abnormal extremal lift.

\end{remark}

\section{Reduced optimal control problem}
\label{sec:PMP_RBVP_ROCP}

In this section, we write the Pontryagin Minimum Principle (PMP) for \eqref{eq:OCP}, derive from it a reduced boundary value problem and a reduced optimal control problem $\mathrm{ROCP}(\lambda)$, and relate the extremals and minimizers of $\mathrm{ROCP}(\lambda)$ and of \eqref{eq:OCP}.

\subsection{Pontryagin Minimum Principle (PMP)}

Introduce adjoint variables $p_x(t)\in\mathbb{R}^n$ and $p_y(t)\in\mathbb{R}^p$, the multiplier $p^0$ and define the Hamiltonian
\begin{multline}
    H(x,y,p_x,p_y,p^0, u) = p_x^\top\left(f_1(x) + F_2(x)u\right)
+ p_y^\top\left(g_1(x) + G_2(x)u\right) \\ + p^0\left(f^0(x) + \tfrac12\, u^\top R u\right).
\label{eq:H}
\end{multline}
According to the PMP \cite{Pontryagin:1962}, if $(x(t), y(t), u(t))$ form an optimal solution of \eqref{eq:OCP} for $t \in [0,T]$, then there exists a scalar $p^0 \geq 0$ and an absolutely continuous function $p = (p_x, p_y): [0,T] \rightarrow \R^n \times \R^p $ satisfying $(p(\cdot), p^0) \neq (0,0)$, such that
\begin{eqnarray}
\label{eq:pmp_conditions_x}
\dot{x}(t) &=& \nabla_{p_x} H[t] , \quad \dot{p}_x(t) = - \nabla_x H[t] , \\[0.5em]
\label{eq:pmp_conditions_y}
\dot{y}(t) &=& \nabla_{p_y} H[t] , \quad \dot{p}_y(t) = - \nabla_y H[t] , \\[0.5em]
\label{eq:pmp_conditions_u}
0 &=& \nabla_u H[t]
\end{eqnarray}
for almost every $t \in [0,T]$ and with $[t] = (x(t), y(t), p_x(t), p_y(t), p^0, u(t))$. Also, we assume the solution not to be abnormal, i.e.\ $p^0 \neq 0$ (see Remark~\ref{rem:normal}) and normalize $p^0$ to $p^0 = 1$.
This minimum formulation of the PMP, with $p^0=1$, is used consistently throughout the paper, including in all examples of Section~\ref{sec:applications}.
Functions $(x(\cdot), y(\cdot), p_x(\cdot), p_y(\cdot), u(\cdot))$ satisfying \eqref{eq:pmp_conditions_x}-\eqref{eq:pmp_conditions_u} and boundary conditions from \eqref{eq:OCP} are called extremals. Since the Hamiltonian does not depend on $y$, the associated adjoint satisfies $\dot p_y=0$, hence
\begin{equation}
p_y(t)\equiv \lambda\in\mathbb{R}^p,
\label{eq:py_const}
\end{equation}
which we will refer in the following as the \emph{cyclic multiplier}. On the other hand, the minimization condition $\nabla_u H=0$ yields
\begin{equation}
\begin{aligned}
& && R u + F_2(x)^\top p_x + G_2(x)^\top \lambda = 0, \\
& \text{i.e.} ~~ &&
u^\star(t) = -R^{-1}\left(F_2(x(t))^\top p_x(t) + G_2(x(t))^\top \lambda\right).
\end{aligned}
\label{eq:u_star}
\end{equation}
Substituting \eqref{eq:u_star} into the state--adjoint dynamics gives the full Hamiltonian boundary value problem in variables $(x,y,p_x)$, parametrized by $\lambda$:
\begin{equation}
\tag{FBVP}
\begin{aligned}
\dot x(t) &=
f_1(x(t)) - F_2(x(t))R^{-1}\left(F_2(x(t))^\top p_x(t) + G_2(x(t))^\top \lambda\right),\\
\dot y(t) &=
g_1(x(t)) - G_2(x(t))R^{-1}\left(F_2(x(t))^\top p_x(t) + G_2(x(t))^\top \lambda\right),\\
\dot p_x(t) &= -\nabla f^0(x(t))
- Df_1(x(t))^\top p_x(t)
- \left(DF_2(x(t))[\,\cdot\,]\,u^\star(t)\right)^\top p_x(t)\\
&\quad - Dg_1(x(t))^\top \lambda
- \left(DG_2(x(t))[\,\cdot\,]\,u^\star(t)\right)^\top \lambda.
\end{aligned}
\label{eq:FBVP_dyn}
\end{equation}
with boundary conditions
\begin{equation}
x(0)=x_0,\quad x(T)=x_f,
\qquad
y(0)=y_0,\quad y(T)=y_f.
\label{eq:FBVP_bc}
\end{equation}
Here $DF_2(x)[\,\cdot\,]$ denotes the Jacobian of the matrix field $F_2$.

\subsection{Reduced problem}

Since the cyclic component enters the optimality system only through the constant parameter $\lambda=p_y$, the $(x,p_x)$-subsystem in \eqref{eq:FBVP_dyn} forms a closed boundary value problem once $\lambda$ is fixed. Let $X_\lambda(x,p_x)$ denote the right-hand side of the $x$- and $p_x$-equations in \eqref{eq:FBVP_dyn}, with $u^\star$ given by \eqref{eq:u_star}. The \emph{reduced boundary value problem} (RBVP) is
\begin{equation}
(\dot x,\dot p_x)=X_\lambda(x,p_x),\qquad x(0)=x_0,\quad x(T)=x_f.
\label{eq:RBVP}
\end{equation}

\begin{definition}
\label{def:ROCP}
Fixing the parameter $\lambda\in\mathbb{R}^p$, the reduced optimal control problem $\mathrm{ROCP}(\lambda)$ is defined by:
\begin{equation*}
\begin{aligned}
\min_{u(\cdot) \in L^\infty([0,T])}\;\; \mathcal{J}_T^\lambda(u,x)
&= \int_0^T \left(
f^0(x) + \tfrac12\, u^\top R u + \lambda^\top \left(g_1(x) + G_2(x)u\right)
\right)\,dt,\\
\text{s.t.}\quad
\dot x &= f_1(x) + F_2(x)u,\\
x(0)&=x_0,\quad x(T)=x_f.
\end{aligned}
\end{equation*}
The normal Hamiltonian of $\mathrm{ROCP}(\lambda)$ is
\begin{equation}
H^\lambda(x,p_x,u) = f^0(x) + \tfrac12\, u^\top R u + \lambda^\top \left( g_1(x) + G_2(x)u \right) + p_x^\top\left(f_1(x) + F_2(x)u\right).
\label{eq:H_ROCP}
\end{equation}
\end{definition}

\begin{proposition}[Equivalence between extremals of ROCP and solutions of RBVP]
\label{prop:ROCP_RBVP}
The pair $(\tilde x(\cdot),\tilde p_x(\cdot))$ satisfies the RBVP \eqref{eq:RBVP} parameterized by $\lambda$ with $u^\star$ defined in \eqref{eq:u_star} if and only if $(\tilde x(\cdot),\tilde p_x(\cdot),u^\star(\cdot))$ is an extremal of  $\mathrm{ROCP}(\lambda)$.
\end{proposition}
\begin{proof} Under the Assumption $p^0 = 1$, and the identification $p_y = \lambda$, $H$ in \eqref{eq:H} coincides with $H^\lambda$ in \eqref{eq:H_ROCP}, and so are the respective derivatives with respect to $u, x, p_x$. This completes the proof.
\end{proof}

For each fixed $\lambda\in\mathbb{R}^p$,
the reduced optimal control problem $\mathrm{ROCP}(\lambda)$ generates, through its Pontryagin extremals, a solution $(x^\lambda(\cdot),p_x^\lambda(\cdot))$ of the reduced boundary value problem \eqref{eq:RBVP}, with control $u^\lambda=u^\star(x^\lambda,p_x^\lambda,\lambda)$. For such a solution, the evolution of the cyclic variable is uniquely determined by
\begin{equation}
\dot y^\lambda(t)
= g_1(x^\lambda(t)) + G_2(x^\lambda(t))\,u^\lambda(t),
\qquad
y^\lambda(0)=y_0.
\label{eq:y_reconstruction}
\end{equation}
Integrating \eqref{eq:y_reconstruction} over $[0,T]$ defines the map
\begin{equation}
\mathcal{Y}_T(\lambda)
:= y^\lambda(T)
= y_0 + \int_0^T
\left(g_1(x^\lambda(t)) + G_2(x^\lambda(t))\,u^\lambda(t)\right)\,dt,
\label{eq:endpoint_map}
\end{equation}
which associates to each $\lambda$ the terminal value of the cyclic variable induced by the corresponding reduced extremal.
The fixed terminal condition $y(T)=y_f$ therefore gives rise to the implicit equation
\begin{equation}
\mathcal{Y}_T(\lambda) = y_f.
\label{eq:lambda_equation}
\end{equation}
In contrast to the free-final case, where $\lambda=0$ follows directly from the transversality condition, the multiplier $\lambda$ must now be computed from the implicit relation \eqref{eq:lambda_equation}.

The next two statements relate the extremals and minimizers of \eqref{eq:OCP} to those of $\mathrm{ROCP}(\lambda)$.
\begin{corollary}[Extremals]
\label{cor:extremals}
Let $\lambda\in\R^p$. A normal extremal $(x,y,p_x,p_y,u)$ of \eqref{eq:OCP} with $p_y\equiv\lambda$ yields a solution $(x,p_x)$ of \eqref{eq:RBVP}, i.e.\ an extremal of $\mathrm{ROCP}(\lambda)$, whose cyclic reconstruction \eqref{eq:y_reconstruction} satisfies $y(T)=y_f$. Conversely, an extremal of $\mathrm{ROCP}(\lambda)$ whose cyclic reconstruction satisfies $y(T)=y_f$ is, together with $p_y\equiv\lambda$, a normal extremal of \eqref{eq:OCP}.
\end{corollary}

\begin{proof}
Both statements follow from Proposition~\ref{prop:ROCP_RBVP}, from \eqref{eq:py_const}, and from the fact that the $y$-equation in \eqref{eq:FBVP_dyn} is the reconstruction equation \eqref{eq:y_reconstruction}.
\end{proof}

\begin{remark}[Extremals and minimizers]
\label{rem:status}
\label{lem:lagrangian_sufficiency}
For every control $u$ and the corresponding solution $(x,y)$ of the dynamics in \eqref{eq:OCP} with $x(0)=x_0$, $y(0)=y_0$, integrating the $y$-equation gives the cost identity
\begin{equation}
\mathcal J_T^\lambda(u,x)=J_T(u,x)+\lambda^\top\big(y(T)-y_0\big).
\label{eq:cost_identity}
\end{equation}
Consequently, a global (resp.\ local, for the $L^\infty$ topology) minimizer of $\mathrm{ROCP}(\lambda)$ whose cyclic reconstruction reaches $y_f$ is a global (resp.\ local) minimizer of \eqref{eq:OCP}, since both problems have the same constraints on $x$ and $\lambda^\top(y_f-y_0)$ is constant. The converse is false in general. Indeed, an optimal solution of \eqref{eq:OCP} is an extremal of $\mathrm{ROCP}(\lambda_T)$ (Corollary~\ref{cor:extremals}), but need not minimize it; in Example~2 below, $\mathrm{ROCP}(\lambda)$ is even unbounded from below when $\lambda_2\neq0$. From the optimization point of view, the terminal condition on $y$ is equivalent to the integral constraint $\int_0^T\big(g_1(x)+G_2(x)u\big)dt=y_f-y_0$, $\mathrm{ROCP}(\lambda)$ is obtained by adjoining only this constraint to the cost with the multiplier $\lambda$, and \eqref{eq:cost_identity} is the corresponding weak duality identity. Average constraints of this kind are also studied in economic model predictive control \cite{Mueller2014}.
\end{remark}

The map $\mathcal Y_T$ is well defined at $\lambda$ if there exists a unique solution $(x^\lambda, p_x^\lambda)$ of \eqref{eq:RBVP}. Additionally, \eqref{eq:lambda_equation} may have several solutions, since several extremals, each with its own multiplier, may reach the same endpoint. Under the assumptions of Proposition~\ref{prop:existence}, the well-definedness of $\mathcal Y_T$ and the uniqueness of the solution of \eqref{eq:lambda_equation} hold locally, that is, for $\lambda$ close to $\bar\lambda$ and for solutions of \eqref{eq:RBVP} close to the reference trim. Globally, the genericity results of \cite{Chitour08}, according to which generic control-affine systems with at least two controls have no minimizing singular trajectories, support the normality assumption (S3): for such systems, optimal solutions have no abnormal extremal lift, and their multiplier is unique. They do not imply the uniqueness of the solution of \eqref{eq:lambda_equation}.

\subsection{Static problems and trims}
\label{sec:lambda_selection}

The \emph{static} problem associated with $\mathrm{ROCP}(\lambda)$ is defined by
\begin{equation}
\begin{aligned}
\min_{(x,u)}\;\; \mathcal{J}_s^\lambda(x,u)
&:= f^0(x) + \tfrac12\, u^\top R u + \lambda^\top \left( g_1(x) + G_2(x)u \right),\\
\text{s.t.}\quad
0 &= f_1(x) + F_2(x)u,
\end{aligned}
\label{eq:static_ROCP}
\end{equation}
where the constraint enforces stationarity of the reduced dynamics. Whenever a KKT point $(x^\circ,u^\circ,p_x^\circ)$ of \eqref{eq:static_ROCP} exists, we call $(x^\circ,u^\circ,p_x^\circ)$ a \emph{static point} of $\mathrm{ROCP}(\lambda)$. Equivalently, $(x^\circ,p_x^\circ)$ is an equilibrium of $X_\lambda$ with $u^\circ=u^\star(x^\circ,p_x^\circ,\lambda)$, because the KKT conditions coincide with the conditions for a point to be an equilibrium of the reduced Hamiltonian system.

\begin{definition}[Trim] \label{def:trim.turnpike}
For given $y_{T/2} \in \R^p$ and $\lambda \in \R^p$, a \emph{trim} of $\mathrm{ROCP}(\lambda)$ is defined by $(x^\circ, y^\circ_{T}(\cdot), u^\circ, p_x^\circ)$, where $(x^\circ,u^\circ,p_x^\circ)$ is a static point of $\mathrm{ROCP}(\lambda)$, i.e.\ the pair $(x^\circ, p_x^\circ)$ is an equilibrium of $X_{\lambda}$ and $u^\circ= -R^{-1}\left(F_2(x^\circ)^\top p_x^\circ + G_2(x^\circ)^\top \lambda\right)$, and $y^\circ_{T}(\cdot)$ is the solution of
\begin{equation}
\dot y^\circ_T(t) = g_1(x^\circ) + G_2(x^\circ)\,u^\circ,
\qquad
y^\circ_T(T/2) = y_{T/2}.
\label{eq:trim}
\end{equation}
\end{definition}
Notice that in contrast to the trim definition of controlled systems \cite{Frazz2005}, the trim is defined here for the optimal control system involving the adjoint variables. Furthermore, the cyclic component of a trim is anchored at the midpoint $T/2$ of the time interval, instead of $t=0$; this is the normalization used in the turnpike estimates in Section~\ref{sec:turnpike}. The trim with respect to which the turnpike occurs, called the \emph{trim turnpike}, is the trim of $\mathrm{ROCP}(\lambda_T)$ selected by the actual cyclic multiplier $\lambda_T$ of the solution, anchored at $y_{T/2}=y_T(T/2)$; see Theorem~\ref{thm:exp_trim_turnpike} and \eqref{eq:trim_2}. When $y(T)$ is free, $\lambda_T=0$ by transversality and the trim turnpike coincides with the one of \cite{Flasskamp2025}. When needed, the dependence on $\lambda$ is indicated, as in $x^\circ(\lambda)$, $u^\circ(\lambda)$, $p_x^\circ(\lambda)$.
We denote by
\begin{equation}
v(x,p_x,\lambda):=g_1(x)+G_2(x)\,u^\star(x,p_x,\lambda),
\label{eq:cyclic_velocity_field}
\end{equation}
the right-hand side of the $y$-equation in \eqref{eq:FBVP_dyn}, which we call the \emph{cyclic velocity}. If $(x^\circ,u^\circ,p_x^\circ)$ is a static point of $\mathrm{ROCP}(\lambda)$, the constant $v(x^\circ,p_x^\circ,\lambda)=g_1(x^\circ)+G_2(x^\circ)u^\circ$ is called its \emph{trim velocity}.

For a prescribed trim velocity $\nu\in\R^p$, we also consider the static problem
\begin{equation}
\mathrm{SP}(\nu):\quad \min_{(x,u)}\; f^0(x)+\tfrac12\,u^\top R\,u \quad\text{s.t.}\ f_1(x)+F_2(x)u=0,\ g_1(x)+G_2(x)u=\nu.
\label{eq:SP_nu}
\end{equation}
Comparing the KKT conditions, $(x,u,p_x,\lambda)$ is a KKT point of $\mathrm{SP}(\nu)$ if and only if $(x,u,p_x)$ is a static point of $\mathrm{ROCP}(\lambda)$ with  $g_1(x)+G_2(x)u=\nu$.

Let us also introduce the \emph{mean cyclic velocity}
\begin{equation}
\nu_T:=\frac{y_f-y_0}{T}\in\R^p .
\label{eq:mean_velocity}
\end{equation}
In the following, we consider a family of problems \eqref{eq:OCP} parametrized by the horizon $T$, with fixed $(x_0,y_0,x_f)$ and a final cyclic condition $y_f=y_f(T)$ which may depend on $T$, and we assume that the limit $\nu_\infty:=\lim_{T\to\infty}\nu_T$ exists (Assumption~\ref{ass:existence}). This holds with $\nu_\infty=0$ when $y_f$ is fixed, and with $\nu_\infty=\nu$ when $y_f(T)=y_0+T\nu+o(T)$, as in Examples~2 and~3 of Section~\ref{sec:applications}.

\begin{definition}[Limit static point and reference trim]
\label{def:reference}
A KKT point $(\bar x,\bar u,\bar p_x,\bar\lambda)$ of $\mathrm{SP}(\nu_\infty)$ is called a \emph{limit static point}. Equivalently, $(\bar x,\bar u,\bar p_x)$ is a static point of $\mathrm{ROCP}(\bar\lambda)$ whose trim velocity is
\begin{equation}
\bar\nu:=g_1(\bar x)+G_2(\bar x)\bar u=\nu_\infty .
\label{eq:nu_bar}
\end{equation}
For any given $y_{T/2}\in\R^p$, the trim $(\bar x,\bar y_T(\cdot),\bar u,\bar p_x)$ of $\mathrm{ROCP}(\bar\lambda)$ with $\bar y_T(T/2)=y_{T/2}$ is called the \emph{reference trim}.
\end{definition}
Throughout the paper, a bar denotes this limit static point, which is fixed once and for all and depends neither on $T$ nor on the boundary data, while a circle denotes the static points and trims of $\mathrm{ROCP}(\lambda)$ for other values of $\lambda$.

\section{Exponential trim turnpike theorem}
\label{sec:turnpike}
In this section, we state and prove the exponential trim turnpike theorem. We first introduce the assumptions, which are local and bear on the limit static point $(\bar x,\bar u,\bar p_x,\bar\lambda)$ of Definition~\ref{def:reference}. We then state the main result, Theorem~\ref{thm:exp_trim_turnpike}, and prove it after establishing some preliminary results. Since the cyclic multiplier $\lambda_T$ of the solution, and with it the reduced problem $\mathrm{ROCP}(\lambda_T)$ and its trim, depend on the time horizon, these preliminary results are established uniformly with respect to $\lambda$ in a neighborhood of $\bar\lambda$; this is what makes all the constants independent of $T$.

\subsection{Main assumptions}
\label{subsec:assumptions_trim}

\begin{assumption}
\emph{Existence.} The limit $\nu_\infty=\lim_{T\to\infty}\nu_T$ exists, and $(\bar x,\bar u,\bar p_x,\bar\lambda)$ is a limit static point (Definition~\ref{def:reference}).
\label{ass:existence}
\end{assumption}
We introduce the following notations
\[
A=D_x\big(f_1+F_2\bar u\big)(\bar x),\quad B=F_2(\bar x),\quad C=D_x\big(g_1+G_2\bar u\big)(\bar x),\quad D=G_2(\bar x),
\]
and $\bar H_{xx}=\nabla^2_{xx}H^{\bar \lambda}(\bar x, \bar p_x, \bar u)$, $\bar H_{xu}=\nabla^2_{xu}H^{\bar \lambda}(\bar x, \bar p_x, \bar u)$, $\bar H_{uu}=R$.

\begin{assumption}
\label{ass:trim}
The limit static point $(\bar x,\bar u,\bar p_x,\bar\lambda)$ satisfies the following conditions.
\begin{itemize}
\item[(H1)] \emph{Positivity.} The matrix $\bar W:=\bar H_{xx}-\bar H_{xu}R^{-1}\bar H_{ux}$ is symmetric positive definite.
\item[(H2)] \emph{Kalman condition on the linearization along the trim.} Since the dynamics in \eqref{eq:OCP} do not depend on $y$, the linearization of the control system \eqref{eq:OCP} along the trim $t\mapsto(\bar x,y_0+t\bar\nu,\bar u)$ is the time-invariant system with matrices
\[
\mathcal A=\begin{pmatrix}A&0\\ C&0\end{pmatrix},\qquad \mathcal B=\begin{pmatrix}B\\ D\end{pmatrix}.
\]
The pair $(\mathcal A,\mathcal B)$ satisfies the condition $\operatorname{rank}\left[\mathcal B,\mathcal A\mathcal B,\dots,\mathcal A^{n+p-1}\mathcal B\right]=n+p$.
\end{itemize}
\end{assumption}

\begin{Lemma} \label{Lem.H2}
Assumption~(H2) holds if and only if the pair $(A,B)$ satisfies the Kalman condition and the $(n+p)\times(n+m)$ matrix
\[
J:=\begin{pmatrix}A&B\\ C&D\end{pmatrix} \quad \text{is surjective.}
\]
\end{Lemma}

\begin{proof}
By the Hautus test, Assumption~(H2) implies
\[
\operatorname{rank} \left[ \begin{pmatrix}s I - A& 0\\ -C& s I\end{pmatrix},  \begin{pmatrix}B \\ D\end{pmatrix}\right] = n+p
\]
for every $s \in \C$. For $s=0$, up to the order and the sign of the columns, this matrix is $\left[J,\ 0\right]$, so that $J$ is surjective. Moreover,
\[
\left[\mathcal B,\mathcal A\mathcal B,\dots,\mathcal A^{n+p-1}\mathcal B\right] = \left[ \begin{pmatrix}B\\ D\end{pmatrix},  \begin{pmatrix}AB \\ CB\end{pmatrix},\dots, \begin{pmatrix}A^{n+p-1}B \\ CA^{n+p-2}B\end{pmatrix} \right].
\]
Since this matrix has rank $n+p$, its first $n$ rows are linearly independent, i.e.\ $\operatorname{rank}\left[B,AB,\dots,A^{n+p-1}B\right]=n$, and by the Cayley--Hamilton theorem $\operatorname{rank}\left[B,AB,\dots,A^{n-1}B\right]=n$: the pair $(A,B)$ satisfies the Kalman condition.\\
Conversely, we have the factorization
\[
\left[\mathcal B,\mathcal A\mathcal B,\dots,\mathcal A^{n+p-1}\mathcal B\right] = \begin{pmatrix}A&B\\ C&D\end{pmatrix} \begin{pmatrix} 0 & \left[ B, AB, \dots, A^{n+p-2}B \right] \\
I_m & \left[ 0_m, 0_m, \dots,  0_m \right]\end{pmatrix},
\]
where $0_m \in \mathrm{M}_m(\R)$ is the zero matrix. If $(A,B)$ satisfies the Kalman condition, then $\operatorname{rank}\left[B,\dots,A^{n+p-2}B\right]=n$, since $n+p-2\geq n-1$, so that the right factor has rank $n+m$, i.e.\ it is surjective onto $\R^{n+m}$. If moreover $J$ is surjective, the product has rank $n+p$, which is (H2).
\end{proof}

\begin{remark} The surjectivity of $J$ implies in particular that the number of controls $m$ can not be less than the number of cyclic variables $p$. In addition, this matrix can be obtained as a Jacobian of constraints in \eqref{eq:SP_nu}. This corresponds to the linear independence constraint qualification (LICQ) for \eqref{eq:SP_nu} and implies the uniqueness of the Lagrange multiplier $(\bar p_x, \bar \lambda)$.
\end{remark}

\begin{Lemma}
\label{lem:hyp_ref}
Let Assumptions~\ref{ass:existence} and~\ref{ass:trim} hold. Then $(\bar x, \bar p_x)$ is a hyperbolic equilibrium of $X_{\bar\lambda}$.
\end{Lemma}

\begin{proof}
Let us denote $z = (x, p_x)$. Linearization of \eqref{eq:RBVP} at $\bar z = (\bar x, \bar p_x)$ takes the form $\delta \dot z = M \delta z$ with $M$ defined as
\[
M=\begin{pmatrix}\tilde A&-S\\ -\bar W&-\tilde A^\top\end{pmatrix},\qquad \tilde A=A-BR^{-1}\bar H_{ux},\qquad S=BR^{-1}B^\top.
\]
Since $M$ is a Hamiltonian matrix obtained by linearization of $\mathrm{ROCP}(\lambda)$, it is invertible and has no eigenvalues on the imaginary axis if $(A,B)$ satisfies the Kalman condition, which holds by Lemma~\ref{Lem.H2}, and $R$ and $\bar W$ are positive definite, see \cite[Lemma 2]{TreZua2015}.
\end{proof}

\subsection{Main result}

\begin{theorem}[Exponential trim turnpike with boundary conditions]
\label{thm:exp_trim_turnpike}
Let Assumptions~\ref{ass:existence} and~\ref{ass:trim} hold. Then there exist constants $T_0, \varepsilon_0, C, \mu>0$ such that, for every $T\geq T_0$, for every $\varepsilon\in(0,\varepsilon_0]$ and for all boundary data $(x_0,y_0,x_f,y_f)$ satisfying
\begin{align}
\|\nu_T-\bar\nu\|\leq\varepsilon,
\label{eq:cyclic.smallness} \\
\|x_0-\bar x\|+\|x_f-\bar x\|\leq\varepsilon,
\label{eq:smallness}
\end{align}
the OCP \eqref{eq:OCP} admits a normal extremal $(x_T,y_T,p_{x,T},\lambda_T,u_T)$ satisfying the following estimates with respect to the trim turnpike, i.e.\ the trim of $\mathrm{ROCP}(\lambda_T)$ given by the static point $(x^\circ(\lambda_T),u^\circ(\lambda_T),p_x^\circ(\lambda_T))$ of Lemma~\ref{lem:trim_branch}, with trim velocity $v^\circ(\lambda_T)=g_1(x^\circ(\lambda_T))+G_2(x^\circ(\lambda_T))\,u^\circ(\lambda_T)$, anchored at $y_T(T/2)$:
\begin{itemize}
\item[(i)] \emph{(Shape variables, adjoint and control.)}
\begin{equation}
\begin{aligned}
\|x_T(t)-x^\circ(\lambda_T)\|+\|p_{x,T}(t)-p_x^\circ(\lambda_T)\|+\|u_T(t)-u^\circ(\lambda_T)\| & \leq \\
& \hspace{-3.5cm} C\varepsilon\left(e^{-\mu t}+e^{-\mu(T-t)}\right), \quad t\in[0,T].
\end{aligned}
\label{eq:turnpike_xu}
\end{equation}
\item[(ii)] \emph{(Cyclic variable.)} With the trim trajectory anchored at the midpoint,
\begin{equation}
y^\circ_T(t):=y_T(T/2)+\Big(t-\frac T2\Big)\,v^\circ(\lambda_T),\qquad t\in[0,T],
\label{eq:trim_2}
\end{equation}
we have
\begin{equation}
\|y_T(t)-y^\circ_T(t)\|\leq C\varepsilon\left(e^{-\mu t}+e^{-\mu(T-t)}\right),\qquad t\in[0,T].
\label{eq:turnpike_y}
\end{equation}
\end{itemize}
The constants $C$ and $\mu$ do not depend on $T$, $\varepsilon$ nor on the boundary data.
\end{theorem}

\begin{remark}
The trim turnpike in Theorem~\ref{thm:exp_trim_turnpike} is well defined: by Proposition~\ref{prop:existence} below, $\lambda_T$ belongs to the ball $\Lambda$ of Lemma~\ref{lem:trim_branch}, on which the static point $(x^\circ(\lambda),u^\circ(\lambda),p_x^\circ(\lambda))$ close to $(\bar x,\bar u,\bar p_x)$ is uniquely defined. Moreover, $\lambda_T$ solves \eqref{eq:lambda_equation}, and it is the only solution of \eqref{eq:lambda_equation} associated with a normal extremal in the $\rho$-tube \eqref{eq:tube}; this relates $\lambda_T$ to the cyclic data $y_0$, $y_f$.
\end{remark}

The proof of Theorem~\ref{thm:exp_trim_turnpike} relies on the following preliminary results. Lemma~\ref{lem:trim_branch} shows that, for $\lambda$ close to $\bar\lambda$, the static points of $\mathrm{ROCP}(\lambda)$ close to $(\bar x,\bar u,\bar p_x)$ form a $C^1$ branch $\lambda\mapsto(x^\circ(\lambda),u^\circ(\lambda),p_x^\circ(\lambda))$ of hyperbolic equilibria of $X_\lambda$. Lemma~\ref{lem:hyperbolic} gives exponential estimates, with constants uniform in $\lambda$, for the solutions of $\dot z=X_\lambda(z)$ which stay close to such an equilibrium. Lemma~\ref{lem:trim_velocity} shows that the trim velocity $\lambda\mapsto v^\circ(\lambda)$ is invertible near $\bar\lambda$, with a quantitative inverse; this is what allows the cyclic condition $y(T)=y_f$ to select the multiplier $\lambda_T$. Proposition~\ref{prop:existence} combines these results: for $T$ large and boundary data close to the limit static point, it provides a normal extremal which stays close to it on the whole horizon, and proves that this extremal is locally optimal and locally unique. Theorem~\ref{thm:exp_trim_turnpike} then follows from Proposition~\ref{prop:existence} and Lemma~\ref{lem:hyperbolic}.

\begin{Lemma}[Branch of trims]

\label{lem:trim_branch}
Let Assumptions~\ref{ass:existence} and~\ref{ass:trim} hold. There exist $\rho_\Lambda>0$, a neighborhood $\mathcal V$ of $(\bar x,\bar u,\bar p_x)$ and a $C^1$ map $\lambda\mapsto(x^\circ(\lambda),u^\circ(\lambda),p_x^\circ(\lambda))$ defined on $\Lambda:=\{\lambda\in\R^p:\|\lambda-\bar\lambda\|\leq\rho_\Lambda\}$, with $(x^\circ,u^\circ,p_x^\circ)(\bar\lambda)=(\bar x,\bar u,\bar p_x)$, such that for every $\lambda\in\Lambda$ %
\begin{itemize}
\item[(a)] $(x^\circ(\lambda),u^\circ(\lambda), p_x^\circ(\lambda))$ is a static point which is the only solution of the KKT system of \eqref{eq:static_ROCP} in $\mathcal V$;
\item[(b)]
$z^\circ(\lambda):=(x^\circ(\lambda),p_x^\circ(\lambda))$ is a hyperbolic equilibrium of $X_\lambda$.
\end{itemize}
\end{Lemma}

The proof is given in Appendix~\ref{app:trim_branch}.

\begin{Lemma}[Exponential estimates near a family of hyperbolic equilibria]
\label{lem:hyperbolic}
Let $\Lambda\subset\R^p$ be compact, let $X:\R^{N}\times\Lambda\to\R^N$ be of class $C^2$ and let $\lambda\mapsto z^\circ(\lambda)$ be continuous, with $X(z^\circ(\lambda),\lambda)=0$ and $D_zX(z^\circ(\lambda),\lambda)$ hyperbolic for every $\lambda\in\Lambda$. There exist $r>0$, $K>0$ and $\mu>0$ such that, for every $\lambda\in\Lambda$, every $T>0$ and every solution $z:[0,T]\to\R^N$ of $\dot z=X(z,\lambda)$ with $\|z(t)-z^\circ(\lambda)\|\leq r$ on $[0,T]$,
\begin{equation}
\|z(t)-z^\circ(\lambda)\|\leq K\Big(\|z(0)-z^\circ(\lambda)\|\,e^{-\mu t}+\|z(T)-z^\circ(\lambda)\|\,e^{-\mu(T-t)}\Big), ~~ t\in[0,T].
\label{eq:hyperbolic_estimate}
\end{equation}
\end{Lemma}

The proof, given in Appendix~\ref{app:hyperbolic}, is a direct dichotomy argument.

\begin{Lemma}[Trim velocity]
\label{lem:trim_velocity}
Let Assumptions~\ref{ass:existence} and~\ref{ass:trim} hold. The map $\lambda \mapsto v^\circ(\lambda)$ defined by
\[
v^\circ(\lambda):=v(z^\circ(\lambda),\lambda)=g_1(x^\circ(\lambda))+G_2(x^\circ(\lambda))\,u^\circ(\lambda),\qquad\lambda\in\Lambda,
\]
satisfies $v^\circ(\bar\lambda)=\bar\nu$, and $Dv^\circ(\bar\lambda)$ is invertible. Consequently, after shrinking $\rho_\Lambda$, there exists $\kappa>0$ such that
\begin{itemize}
\item[(a)] $\|v^\circ(\lambda)-v^\circ(\lambda')\|\geq\kappa\|\lambda-\lambda'\|$ for all $\lambda,\lambda'\in\Lambda$; in particular, $v^\circ$ is injective on $\Lambda$, and its inverse $(v^\circ)^{-1}:v^\circ(\Lambda)\to\Lambda$ is Lipschitz with constant $\kappa^{-1}$;
\item[(b)] for every $\lambda'\in\Lambda$ and $\sigma>0$ such that the closed ball $\bar B(\lambda',\sigma)$ is contained in $\Lambda$, we have $\bar B\big(v^\circ(\lambda'),\kappa\sigma\big)\subset v^\circ\big(\bar B(\lambda',\sigma)\big)$.
\end{itemize}
\end{Lemma}
The proof is given in Appendix~\ref{app:trim_branch_2}. %

\begin{remark}
\label{rem:SP}
By the equivalence stated after \eqref{eq:SP_nu}, for $\nu\in v^\circ(\Lambda)$, the KKT point of $\mathrm{SP}(\nu)$ close to $(\bar x,\bar u)$ is $(x^\circ,u^\circ)(\lambda)$ with $\lambda=(v^\circ)^{-1}(\nu)$, the multiplier of the velocity constraint.
\end{remark}

In the sequel, $\Lambda$ denotes the ball given by Lemma~\ref{lem:trim_velocity}, on which the conclusions of Lemma~\ref{lem:trim_branch} hold, $\ell\geq1$ is a common Lipschitz constant on $\Lambda$ of the $C^1$ maps $x^\circ$, $u^\circ$, $p_x^\circ$, $z^\circ$, and $r$, $K$, $\mu$ are the constants of Lemma~\ref{lem:hyperbolic} applied to $X(z,\lambda)=X_\lambda(z)$ and to the branch $z^\circ$ on $\Lambda$, as constructed in Appendix~\ref{app:hyperbolic}.
The cyclic boundary data enter only through the mean cyclic velocity $\nu_T$, which must be close to the velocity of the trim followed during most of $[0,T]$; since $\nu_T\to\bar\nu$, condition \eqref{eq:cyclic.smallness} holds for $T$ large enough. For $\rho>0$, we say that a normal extremal $(x,y,p_x,\lambda,u)$ of \eqref{eq:OCP} (with $p_y\equiv\lambda$) lies in the \emph{$\rho$-tube} of the reference trim if
\begin{equation}
\sup_{t\in[0,T]}\Big(\|x(t)-\bar x\|+\|p_x(t)-\bar p_x\|\Big)+\|\lambda-\bar\lambda\|\leq\rho .
\label{eq:tube}
\end{equation}

\begin{proposition}[Existence and localization]
\label{prop:existence}
Let Assumptions~\ref{ass:existence} and~\ref{ass:trim} hold. There exist $\varepsilon_0>0$, $\rho>0$, $T_0>0$, $C_0>0$ and a convex neighborhood $\mathcal K$ of $(\bar x,\bar u)$ such that, for every $T\geq T_0$, every $\varepsilon\in(0,\varepsilon_0]$ and all boundary data $(x_0,y_0,x_f,y_f)$ satisfying \eqref{eq:cyclic.smallness}--\eqref{eq:smallness}, the following holds.
\begin{itemize}
\item[(i)] \emph{(Existence and localization.)} Problem \eqref{eq:OCP} has a normal extremal $(x_T,y_T,p_{x,T},\lambda_T,u_T)$ such that
\begin{equation}
\sup_{t\in[0,T]}\Big(\|x_T(t)-\bar x\|+\|p_{x,T}(t)-\bar p_x\|+\|u_T(t)-\bar u\|\Big)+\|\lambda_T-\bar\lambda\|\leq C_0\,\varepsilon .
\label{eq:localization}
\end{equation}
\item[(ii)] \emph{(Local uniqueness.)} It is the only normal extremal of \eqref{eq:OCP} lying in the $\rho$-tube \eqref{eq:tube}. Moreover, $(x_T,y_T,u_T)$ has no abnormal extremal lift, and its normal extremal lift is unique.
\item[(iii)] \emph{(Local optimality.)} $(x_T,y_T,u_T)$ is a strict local minimizer of \eqref{eq:OCP}: its cost is strictly smaller than the cost of any other admissible trajectory $(x,y,u)$ with $(x(t),u(t))\in\mathcal K$ for a.e.\ $t\in[0,T]$.
\end{itemize}
The constants $\varepsilon_0$, $\rho$, $T_0$, $C_0$ and the set $\mathcal K$ do not depend on $T$ nor on the boundary data, and $\rho$ can be chosen such that $\rho\leq\rho_\Lambda/2$ and $(1+\ell)\rho\leq r$.
\end{proposition}

The proof is given in Appendix~\ref{app:existence}. Local optimality and uniqueness follow from the strict convexity of the Hamiltonian with respect to $(x,u)$ near the limit static point, which is a consequence of (H1).

\begin{remark}[Local and global minimizers]
\label{rem:existence_local}
\label{rem:status_thm}
The localization \eqref{eq:localization} is a conclusion, not an assumption, and by item (ii) any normal extremal in the $\rho$-tube is the one of item (i). Under local assumptions only, nothing can be said about minimizers far from the reference trim: modifying the data of \eqref{eq:OCP} outside a neighborhood of $(\bar x,\bar u)$ does not affect Assumption~\ref{ass:trim}, but may create cheaper trajectories. The estimates of Theorem~\ref{thm:exp_trim_turnpike} thus apply to a \emph{global} minimizer of \eqref{eq:OCP} only if it has a normal extremal lift in the $\rho$-tube; this is the case for convex problems (Corollary~\ref{cor:convex}). Under the global assumptions of \cite[Theorem~2]{Tre:23}, the global minimizer stays close to the trim on most of the horizon, and the local estimates can be expected to apply on a sub-interval; a unique global static minimizer is not sufficient in general \cite{Tre:23}.
\end{remark}

\begin{proof}[Proof of Theorem~\ref{thm:exp_trim_turnpike}]
Throughout, $C_1,C_2,\ldots$ denote constants independent of $T$, $\varepsilon$ and of the boundary data; $\ell$, $r$, $K$, $\mu$ are fixed as after Lemma~\ref{lem:hyperbolic}, and $\ell_u$, $L_v$ are Lipschitz constants of $u^\star$ and $v$ on the compact set $\mathcal O:=\{(z,\lambda):\lambda\in\Lambda,\ \|z-z^\circ(\lambda)\|\leq r\}$. Recall from Proposition~\ref{prop:existence} and its proof that $C_0\varepsilon_0\leq\rho$ and $(1+\ell)\rho\leq r$.

\emph{Step 1 (Exponential estimates for $(x,p_x,u)$).} By Corollary~\ref{cor:extremals}, $z_T:=(x_T,p_{x,T})$ solves $\dot z=X_{\lambda_T}(z)$ on $[0,T]$. By \eqref{eq:localization},
\[
\|z_T(t)-z^\circ(\lambda_T)\|\leq\|z_T(t)-z^\circ(\bar\lambda)\|+\|z^\circ(\bar\lambda)-z^\circ(\lambda_T)\|\leq(1+\ell)C_0\varepsilon\leq r,\qquad t\in[0,T],
\]
and Lemma~\ref{lem:hyperbolic} gives
\[
\|z_T(t)-z^\circ(\lambda_T)\|\leq K(1+\ell)C_0\varepsilon\left(e^{-\mu t}+e^{-\mu(T-t)}\right),\qquad t\in[0,T].
\]
Since $u_T(t)=u^\star(z_T(t),\lambda_T)$ and $u^\circ(\lambda_T)=u^\star(z^\circ(\lambda_T),\lambda_T)$, \eqref{eq:turnpike_xu} follows with $C:=(\sqrt2+\ell_u)K(1+\ell)C_0$.

\emph{Step 2 (Cyclic variable).} Set $\delta y(t):=y_T(t)-y^\circ_T(t)$. Then, by \eqref{eq:trim_2}, $\delta y(T/2)=0$. Moreover, since $\delta\dot y(t)=v(z_T(t),\lambda_T)-v(z^\circ(\lambda_T),\lambda_T),$ we have by Step~1:
\[
\qquad \|\delta\dot y(t)\|\leq C_1\varepsilon\left(e^{-\mu t}+e^{-\mu(T-t)}\right),
\]
with $C_1:=L_vK(1+\ell)C_0$. Let $t\in[T/2,T]$. Since $\delta y(T/2)=0$, we integrate this bound from $T/2$ to $t$, using
\[
\int_{T/2}^{t}e^{-\mu s}\,ds\leq \frac1\mu e^{-\mu T/2}\leq\frac1\mu e^{-\mu(T-t)},\qquad
\int_{T/2}^{t}e^{-\mu(T-s)}\,ds\leq\frac1\mu e^{-\mu(T-t)},
\]
where we used $T-t\leq T/2$ in the first bound. Hence $\|\delta y(t)\|\leq\frac{2C_1}{\mu}\varepsilon\left(e^{-\mu t}+e^{-\mu(T-t)}\right)$ on $[T/2,T]$, and symmetrically on $[0,T/2]$. This is \eqref{eq:turnpike_y}, after enlarging $C$. This finishes the proof.
\end{proof}

\begin{remark}[A priori estimates]
\label{rem:apriori}
The proof of Theorem~\ref{thm:exp_trim_turnpike} only uses that the extremal is normal and lies in the $\rho$-tube \eqref{eq:tube}: then $\lambda\in\Lambda$, since $\rho\leq\rho_\Lambda/2$, and $\|z(t)-z^\circ(\lambda)\|\leq(1+\ell)\rho\leq r$ on $[0,T]$, so that Lemma~\ref{lem:hyperbolic} applies. Hence, for every $T>0$, all boundary data and every normal extremal $(x,y,p_x,\lambda,u)$ of \eqref{eq:OCP} lying in the $\rho$-tube, the estimates \eqref{eq:turnpike_xu}--\eqref{eq:turnpike_y} hold for the trim of $\mathrm{ROCP}(\lambda)$ anchored at $y(T/2)$, with $C\varepsilon$ replaced by a constant $C'$ which does not depend on $T$, on the boundary data nor on the extremal. The same holds for Corollary~\ref{cor:multiplier}(i) below, with $\varepsilon$ replaced by $1$. For $T\geq T_0$ and boundary data satisfying \eqref{eq:cyclic.smallness}--\eqref{eq:smallness}, such an extremal coincides with the one of Theorem~\ref{thm:exp_trim_turnpike}, by Proposition~\ref{prop:existence}(ii).
\end{remark}

\begin{corollary}[Cyclic multiplier and anchor]
\label{cor:multiplier}
Under the assumptions of Theorem~\ref{thm:exp_trim_turnpike}, there exist $C>0$ and $T_1\geq T_0$ such that, for every $\varepsilon\in(0,\varepsilon_0]$ and all boundary data satisfying \eqref{eq:cyclic.smallness}--\eqref{eq:smallness}, the extremal $(x_T,y_T,p_{x,T},\lambda_T,u_T)$ of Theorem~\ref{thm:exp_trim_turnpike} satisfies:
\begin{itemize}
\item[(i)] for every $T\geq T_0$,
\begin{gather*}
\|\nu_T-v^\circ(\lambda_T)\|\leq\frac{C\varepsilon}T,\qquad \Big\|y_T(T/2)-\tfrac12(y_0+y_f)\Big\|\leq C\varepsilon,\\
\|\lambda_T-\bar\lambda\|\leq C\Big(\|\nu_T-\bar\nu\|+\frac\varepsilon T\Big);
\end{gather*}
\item[(ii)] for every $T\geq T_1$, $\nu_T$ belongs to $v^\circ(\Lambda)$ and
\[
\big\|\lambda_T-(v^\circ)^{-1}(\nu_T)\big\|\leq\frac{C\varepsilon}T.
\]
\end{itemize}
\end{corollary}

\begin{proof}
(i) Integrating the $y$-equation of \eqref{eq:FBVP_dyn} and using $y_T(T)-y_T(0)=T\nu_T$,
\[
\nu_T-v^\circ(\lambda_T)=\frac1T\int_0^T\Big(v(z_T(t),\lambda_T)-v(z^\circ(\lambda_T),\lambda_T)\Big)\,dt=\frac1T\int_0^T\delta\dot y(t)\,dt,
\]
with $\delta y$ as in Step~2 of the proof of Theorem~\ref{thm:exp_trim_turnpike}. The bound on $\delta\dot y$ obtained there gives $\|\nu_T-v^\circ(\lambda_T)\|\leq 2C_1\varepsilon/(\mu T)$. Moreover, $I_0:=\int_0^{T/2}\delta\dot y\,dt$ and $I_1:=\int_{T/2}^T\delta\dot y\,dt$ are bounded by $2C_1\varepsilon/\mu$, and $y_T(T/2)=y_0+\tfrac T2v^\circ(\lambda_T)+I_0$, $y_f=y_T(T/2)+\tfrac T2v^\circ(\lambda_T)+I_1$, hence $y_T(T/2)=\tfrac12(y_0+y_f)+\tfrac12(I_0-I_1)$. Finally, since $\lambda_T\in\Lambda$, Lemma~\ref{lem:trim_velocity}(a) gives $\|\lambda_T-\bar\lambda\|\leq\kappa^{-1}\|v^\circ(\lambda_T)-v^\circ(\bar\lambda)\|\leq\kappa^{-1}\big(\|v^\circ(\lambda_T)-\nu_T\|+\|\nu_T-\bar\nu\|\big)$.

(ii) We have $\|\lambda_T-\bar\lambda\|\leq C_0\varepsilon\leq\rho\leq\rho_\Lambda/2$. Let $\sigma:=\rho_\Lambda-\rho$, so that $\bar B(\lambda_T,\sigma)\subset\Lambda$, and let $T_1\geq T_0$ be such that $2C_1\varepsilon_0/(\mu T_1)\leq\kappa\sigma$. For $T\geq T_1$, (i) and Lemma~\ref{lem:trim_velocity}(b) give $\nu_T\in v^\circ(\bar B(\lambda_T,\sigma))\subset v^\circ(\Lambda)$, and Lemma~\ref{lem:trim_velocity}(a) gives $\|\lambda_T-(v^\circ)^{-1}(\nu_T)\|\leq\kappa^{-1}\|v^\circ(\lambda_T)-\nu_T\|\leq 2C_1\varepsilon/(\kappa\mu T)$.
\end{proof}

\begin{remark}[Comparison with the linear turnpike theorem \cite{Tre:23}]
\label{rem:comparison_linear}
Three references can be used to describe the solution on long horizons.
(a) The limit static point $(\bar x,\bar u,\bar p_x,\bar\lambda)$, which does not depend on $T$. By Corollary~\ref{cor:multiplier}(i), $\lambda_T-\bar\lambda=O(\|\nu_T-\bar\nu\|+\varepsilon/T)$, so that the trim turnpike differs from the reference trim by the same order in the shape variables, and by $t\,(\nu_T-\bar\nu)+O(1)$ in the cyclic variables. A fixed reference thus only gives estimates up to an error which vanishes as $T\to\infty$ at the rate at which $\nu_T\to\nu_\infty$.
(b) The solution of $\mathrm{SP}(\nu_T)$, which is the reference of \cite{Tre:23}. Under global assumptions (uniqueness, strict dissipativity, controllability), \cite[Theorem~2]{Tre:23} gives the linear turnpike estimates $\|x_T(t)-x^T_{\mathrm{st}}\|\leq C(1/T+e^{-\mu' t}+e^{-\mu'(T-t)})$, $\|y_T(t)-(y_0+t\nu_T)\|\leq C$ and $\|\lambda_T-\lambda^T_{\mathrm{st}}\|\leq C/T$, where $x^T_{\mathrm{st}}$ is the minimizer of $\mathrm{SP}(\nu_T)$ and $\lambda^T_{\mathrm{st}}$ the multiplier of its velocity constraint. When this minimizer is the KKT point close to $(\bar x,\bar u)$, we have $\lambda^T_{\mathrm{st}}=(v^\circ)^{-1}(\nu_T)$ and $x^T_{\mathrm{st}}=x^\circ(\lambda^T_{\mathrm{st}})$ (Remark~\ref{rem:SP}), and the last estimate is Corollary~\ref{cor:multiplier}(ii).
(c) The trim turnpike, selected by $\lambda_T$ and anchored at $T/2$ (Theorem~\ref{thm:exp_trim_turnpike}). The $1/T$ term and the merely bounded error on the cyclic variables in (b) come from the choice of the reference: $\lambda^T_{\mathrm{st}}$ differs from $\lambda_T$ by $O(1/T)$, and $\nu_T$ differs from $v^\circ(\lambda_T)$ by $O(1/T)$, which integrates to $O(1)$ on $[0,T]$. Selecting the trim by $\lambda_T$ and anchoring it at the midpoint restores exponential convergence for all components. When $y(T)$ is free, $\lambda_T=0$ and Theorem~\ref{thm:exp_trim_turnpike} reduces to the trim turnpike of \cite{Flasskamp2025}.
\end{remark}

\begin{corollary}[Convex problems]
\label{cor:convex}

Assume that $f_1$ and $g_1$ are affine, that $F_2$ and $G_2$ are constant and that $f^0$ is convex. Then any normal extremal of \eqref{eq:OCP} is a global minimizer. Under the assumptions of Proposition~\ref{prop:existence}, the extremal of Proposition~\ref{prop:existence}(i) is therefore the unique global minimizer of \eqref{eq:OCP}.

\end{corollary}

\begin{proof}
Problem \eqref{eq:OCP} then has linear dynamics and a convex cost in $(x,u)$, strictly convex in $u$. For such problems, the normal Pontryagin conditions are sufficient for global optimality, since the Hamiltonian is convex in $(x,u)$ (Mangasarian-type sufficient condition, see e.g.\ \cite{trelat2024control}). Since $x$ depends affinely on $u$ and the cost is strictly convex in $u$, the global minimizer is unique.
\end{proof}

\begin{remark}[Linear--quadratic case]
\label{rem:LQ_global}
If $f_1,g_1$ are linear, $F_2,G_2$ constant and $f^0$ quadratic, the reduced Hamiltonian vector field is affine, with a linear part $M$ independent of $\lambda$. Lemma~\ref{lem:hyperbolic} then holds without the restriction $\|z(t)-z^\circ(\lambda)\|\leq r$, so that the a priori estimates of Remark~\ref{rem:apriori} hold for every normal extremal, as illustrated by Example~1.
\end{remark}

\section{Examples}
\label{sec:applications}

This section illustrates the exponential trim turnpike theorem in three representative examples i.e. 1) a linear-quadratic (LQ) problem exhibiting a velocity turnpike (see \cite{Faulwasser2021}), which was also analyzed in \cite{Tre:23} and \cite{Dario2021}; 2) a nonlinear example with two cyclic variables and non-static trims; and 3) the Kepler problem considered in \cite{Flasskamp2025}.
The numerical solutions of Examples~2 and~3 were computed in two independent ways, which give the same results: by an indirect method, solving the full Pontryagin boundary value problem \eqref{eq:FBVP_dyn}--\eqref{eq:FBVP_bc} by collocation (\texttt{scipy.integrate.solve\_bvp}) with the cyclic multiplier $\lambda$ as an unknown parameter, and by a direct method with the package OptimalControl.jl \cite{caillau2024optimalcontrol}. For Example~1, the numerical solution of \eqref{eq:FBVP_dyn}--\eqref{eq:FBVP_bc} agrees with the exact solution given below. In all examples, the trim turnpike is computed as the equilibrium of the reduced Hamiltonian system \eqref{eq:RBVP} for $\lambda=\lambda_T$, its cyclic component being anchored at $T/2$ as in \eqref{eq:trim_2}, and the decay rates are measured on the distance to this trim.

\subsection{Example 1: Linear-quadratic case}
\label{subsec:ex_LQ_linear_turnpike}

We consider the following problem
\begin{equation}
\label{eq:ex1_ocp}
\left\{
\begin{aligned}
\min_{u(\cdot)}\quad & J_T(x,u) := \int_0^{T} \left(x^2(t)+u^2(t)\right)\,dt,\\
& \dot x(t)=u(t),\qquad x(0)=1,\quad x(T)=2,\\
& \dot y(t)=x(t),\qquad y(0)=0,\quad y(T)=y_f,
\end{aligned}
\right.
\end{equation}
where $T>0$ and $y_f\in\R$ are given. Here $x\in\mathbb{R}$ is the non-cyclic state, $y\in\mathbb{R}$ is cyclic, and $u\in\mathbb{R}$. %
By the PMP, $H=x^2+u^2+p_xu+p_yx$, so that $u^\star=-p_x/2$, $\dot p_x=-(2x+p_y)$ and $p_y\equiv\lambda$. Hence $\ddot x-x=\lambda/2$, and
\begin{equation}
\label{eq:ex1_x_general}
x(t)= a\ e^t + b\ e^{-t} -\frac{\lambda}{2},
\qquad
u(t )= \dot x(t) = a\ e^t - b\ e^{-t}.
\end{equation}
The boundary conditions $x(0)=1$ and $x(T)=2$ imply
\begin{equation}
\label{eq:ex1_a}
a + b = 1 + \frac{\lambda}{2},
\qquad a\ e^T + b\ e^{-T} -\frac{\lambda}{2} = 2.
\end{equation}
Finally, the constraint $y(T) = y_f$ is equivalent to
\begin{equation}
\label{eq:ex1_integral_constraint}
\int_0^T x(t)\,dt = y_f,
\end{equation}
since $y(T)=y(0)+\int_0^T x(t)\,dt$ and $y(0)=0$. By integration we obtain
\begin{equation}
\label{eq:y_1_T}
a\ e^T - b\ e^{-T} - a + b - \frac{\lambda T}{2} = y_f.
\end{equation}

The three equations \eqref{eq:ex1_a}--\eqref{eq:y_1_T} form a linear system in $(a,b,\lambda)$. %
Using the first two equations to eliminate $a$ and $b$, the third one becomes linear in $\lambda$, and one obtains
\begin{equation}
\label{eq:ex1_lambda_b}
\lambda_T = \frac{2\left(3\tanh(T/2)-y_f\right)}{T-2\tanh(T/2)}, ~~
a_T = \frac{2+\frac{\lambda_T}{2}-\left(1+\frac{\lambda_T}{2}\right)e^{-T}}{2\sinh T}, ~~ b_T = 1+\frac{\lambda_T}{2}-a_T,
\end{equation}
for every $T>0$ such that $T\neq 2\tanh(T/2)$, in particular for all $T\geq 2$. Since $\tanh(T/2)=1+O(e^{-T})$, this gives
\begin{equation}
\label{eq:ex1_lambda_asymp}
\lambda_T=\frac{2(3-y_f)}{T-2}+O\!\left(e^{-T}\right),\qquad a_T=O\!\left(e^{-T}\right),\qquad b_T=1+\frac{\lambda_T}{2}+O\!\left(e^{-T}\right),
\end{equation}
with constants in the $O(\cdot)$ terms depending only on $y_f$.

The resulting optimal pair $(x_T,u_T)$ is then given by \eqref{eq:ex1_x_general} with $a = a_T$ and $b = b_T$.

\emph{Relation with Section~\ref{sec:turnpike}.} %
In the notation of Section~\ref{sec:turnpike}, the trims are $x^\circ(\lambda)=-\lambda/2$, $u^\circ(\lambda)=0$, with velocity $v^\circ(\lambda)=-\lambda/2$, so that $(v^\circ)^{-1}(\nu)=-2\nu$, and the solution of $\mathrm{SP}(\nu)$ (Remark~\ref{rem:SP}) is $x=\nu$. Assumption~\ref{ass:trim} holds at every static point: $\bar W=2$, and $\begin{pmatrix}A&B\\ C&D\end{pmatrix}=\begin{pmatrix}0&1\\ 1&0\end{pmatrix}$ has rank $2$. With $\nu_T=y_f/T$, \eqref{eq:ex1_lambda_asymp} gives $\lambda_T-(v^\circ)^{-1}(\nu_T)=\frac{2(3-y_f)}{T-2}+\frac{2y_f}{T}+O(e^{-T})=O(1/T)$, in agreement with Corollary~\ref{cor:multiplier}(ii). The boundary data $x_0=1$, $x_f=2$ are not close to the trim, but the problem is convex and linear--quadratic, so Corollary~\ref{cor:convex} and Remark~\ref{rem:LQ_global} apply. The extremal is the unique global minimizer and the exponential estimates hold without smallness condition. For fixed $y_f$, $\nu_\infty=0$ and the limit static point is $(\bar x,\bar u,\bar p_x,\bar\lambda)=(0,0,0,0)$; non-static reference trims appear when $y_f$ grows linearly with $T$, as for $y_f=0.5\,T$ in Figure~\ref{fig:LQ-2}.

By \eqref{eq:ex1_x_general} and \eqref{eq:ex1_lambda_asymp}, $x_T(t)-x^\circ(\lambda_T)=a_Te^t+b_Te^{-t}$ and $u_T(t)=a_Te^t-b_Te^{-t}$, with $a_T=O(e^{-T})$ and $b_T=O(1)$, which gives \eqref{eq:turnpike_xu} with $\mu=1$; integrating $\dot y_T-v^\circ(\lambda_T)=x_T-x^\circ(\lambda_T)$ from $T/2$ gives \eqref{eq:turnpike_y}. The problem appears as the motivating example in \cite[Section~1.2.2]{Tre:23}, so we compare the respective turnpikes. The linear turnpike reference of \cite{Tre:23}, $x^T_{\mathrm{st}}=\nu_T$ and $y_0+t\nu_T$, only gives errors of order $1/T$ on $x$ and of order one on $y$, as predicted by Remark~\ref{rem:comparison_linear} (Figures~\ref{fig:LQ-1} and~\ref{fig:LQ-2}).

\begin{figure}[ht]
\centering
\includegraphics[width=\textwidth]{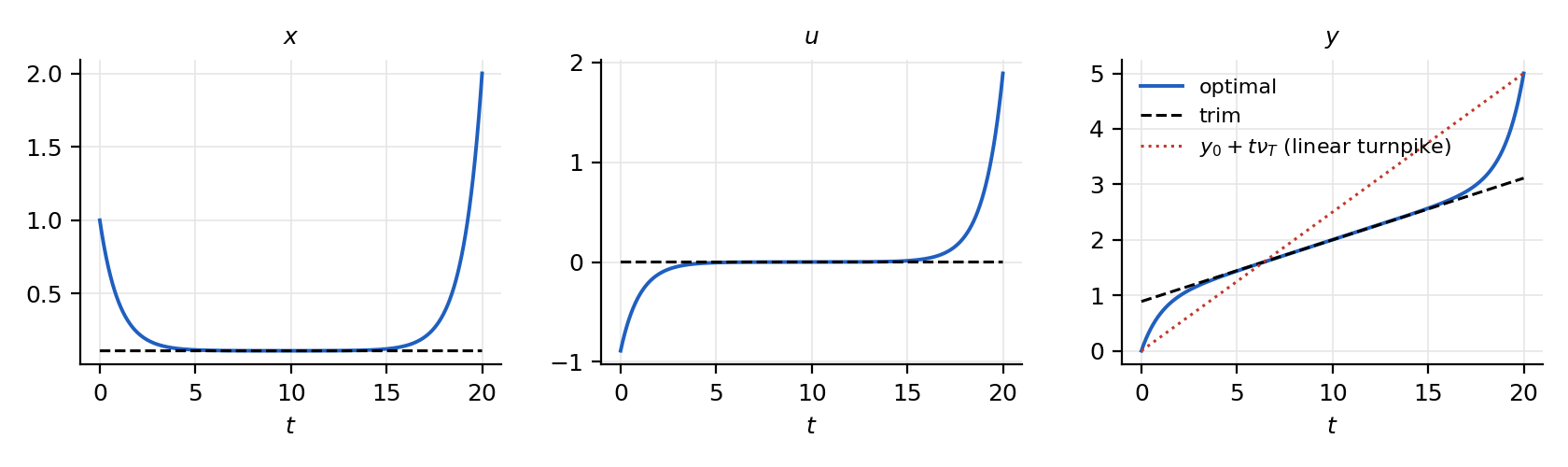}
\caption{Example 1 ($T=20$, $y_f=5$). Optimal solution (blue), trim selected by $\lambda_T$ (dashed black, the cyclic trim being anchored at $T/2$), and the linear-turnpike reference $y_0+t\nu_T$ of \cite{Tre:23} (dotted red).}
\label{fig:LQ-1}
\end{figure}

\begin{figure}[ht]
\centering
\includegraphics[width=\textwidth]{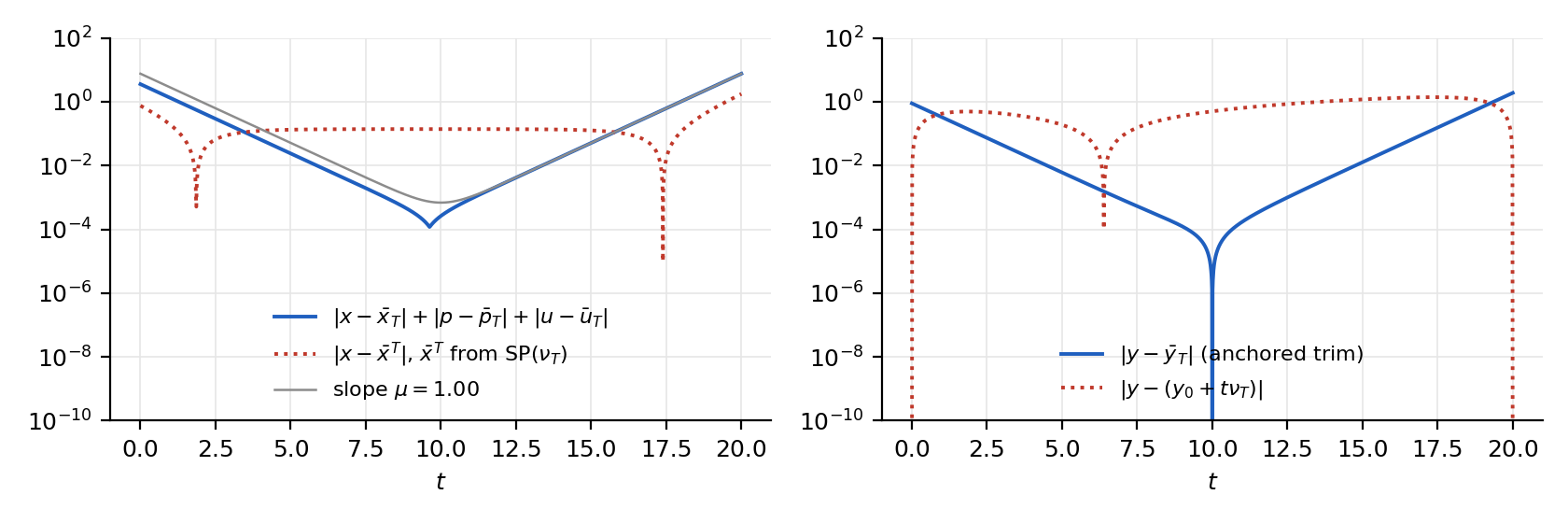}\\[2pt]
\includegraphics[width=0.48\textwidth]{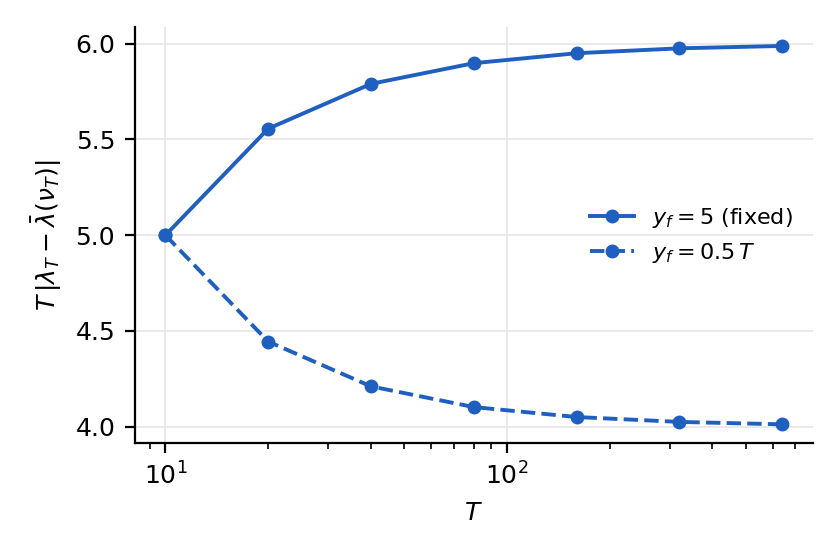}
\vspace{-0.3cm}
\caption{Example 1. Top: distance to the trim selected by $\lambda_T$ (blue, exponential with rate $\mu=1$) and to the reference of the linear turnpike theorem \cite{Tre:23} built on $\mathrm{SP}(\nu_T)$ (dotted red: plateau of order $1/T$ for $x$, error of order one for $y$). Bottom: $T\,|\lambda_T-(v^\circ)^{-1}(\nu_T)|$ for fixed $y_f=5$ and for $y_f=0.5\,T$, illustrating Corollary~\ref{cor:multiplier}(ii).}
\label{fig:LQ-2}
\end{figure}

\subsection{Example 2: Nonlinear case with two cyclic variables}
\label{subsec:ex_NLQ_case}

We consider the nonlinear problem
\begin{equation}
\label{eq:NLQ_full_problem}
\left\{
\begin{aligned}
\min_{u(\cdot)}\quad
&J_T(u)= \tfrac{1}{2}\int_0^{T}\left((x_2(t)-x_d)^2+u_1(t)^2+u_2(t)^2\right)\,dt,\\
\text{s.t.}\quad
&\dot x_1(t)=u_1(t), \qquad\ \ x_1(0)=x_{10},\ \ x_1(T)=x_1^T,\\
&\dot x_2(t)=u_2(t),\qquad\ \ x_2(0)=x_{20},\ \ x_2(T)=x_2^T,\\
&\dot x_3(t)=x_2(t)^2\,u_1(t),\qquad x_3(0)=x_{30},\ \ x_3(T)=x_3^T,
\end{aligned}
\right.
\end{equation}
with a target value $x_d\in\R\setminus\{0\}$. The variables $x_1$ and $x_3$ appear neither in the dynamics nor in the cost. Hence the shape variable is $x:=x_2\in\R$, the cyclic variable is $y=(y_1,y_2):=(x_1,x_3)\in\R^2$. %

\medskip
\noindent\emph{Reduced problem and trims.}
With $p_y\equiv\lambda=(\lambda_1,\lambda_2)$, the minimization condition gives $u_1=-(\lambda_1+\lambda_2x^2)$ and $u_2=-p_x$, and the reduced Hamiltonian system \eqref{eq:RBVP} reads
\begin{equation}
\dot x=-p_x,\qquad \dot p_x=-(x-x_d)+2\lambda_2\,x\,(\lambda_1+\lambda_2x^2).
\label{eq:ex2_RBVP}
\end{equation}
Its equilibria are $p_x^\circ=0$ and the roots $x^\circ$ of $(x^\circ-x_d)=2\lambda_2x^\circ(\lambda_1+\lambda_2(x^\circ)^2)$, with $u^\circ=(-(\lambda_1+\lambda_2(x^\circ)^2),0)$ and trim velocity $v^\circ(\lambda)=u^\circ_1\,(1,(x^\circ)^2)$.

\medskip
\noindent\emph{Static problem with prescribed velocity and assumptions.}
For a reference velocity $\bar\nu=(\nu_1,\nu_2)$ with $\nu_1\neq0$ and $\nu_2/\nu_1>0$, the constraints of $\mathrm{SP}(\bar\nu)$ give $\bar u=(\nu_1,0)$ and $\bar x^2=\nu_2/\nu_1$; we take the root $\bar x$ with the sign of $x_d$. The KKT conditions of $\mathrm{SP}(\bar\nu)$ (Remark~\ref{rem:SP}) give the multipliers
\[
\bar p_x=0,\qquad \bar\lambda_2=\frac{x_d-\bar x}{2\nu_1\bar x},\qquad \bar\lambda_1=-\nu_1-\bar\lambda_2\bar x^2 .
\]
Assumption~\ref{ass:trim} reads as follows. With $A=0$, $B=(0\ 1)$, $C=(0,\ 2\bar x\nu_1)^\top$ and $D=\big(\begin{smallmatrix}1&0\\ \bar x^2&0\end{smallmatrix}\big)$, the matrix $\big(\begin{smallmatrix}A&B\\ C&D\end{smallmatrix}\big)$ has rank $3=n+p$ if and only if $\bar x\nu_1\neq0$, and $(A,B)$ is controllable, so (H2) holds. Moreover,
\[
\bar W=1+2\bar\lambda_2\nu_1-4\bar\lambda_2^2\bar x^2=\frac{x_d}{\bar x}-\frac{(x_d-\bar x)^2}{\nu_1^2},
\]
so (H1) holds when $\bar x$ is close enough to $x_d$, i.e.\ when the prescribed velocity is close to $\nu_1(1,x_d^2)$. For the data below, $x_d=1$ and $\bar\nu=(0.2,0.22)$, one finds $\bar x\approx1.0488$, $\bar\lambda\approx(-0.0720,-0.1163)$ and $\bar W\approx0.894>0$. Note that $\mathrm{SP}(\bar\nu)$ has a second KKT point $-\bar x$, at which $\bar W<0$: the assumptions select one of the two trims, which illustrates the local nature of Theorem~\ref{thm:exp_trim_turnpike}.

For the simulations we take $T=50$, $x_d=1$, $x_2(0)=0$, $x_2(T)=2$, $x_1(0)=x_3(0)=0$, $x_1(T)=0.2\,T$ and $x_3(T)=0.22\,T$, so that $\nu_T=\bar\nu=(0.2,0.22)$ for every $T$.
We obtain $\lambda_T\approx(-0.1323,-0.0632)$, $x^\circ_T\approx1.0258$ and the trim velocity $v^\circ(\lambda_T)\approx(0.1988,0.2092)$, which differs from $\nu_T$ by $O(1/T)$. The linearization of \eqref{eq:ex2_RBVP} at the trim has eigenvalues $\pm0.979$, and the decay rate measured on the numerical solution is $0.979$ (Figure~\ref{fig:NLQ-1}). Figure~\ref{fig:lambda} (left) shows that $T\,\|\lambda_T-(v^\circ)^{-1}(\nu_T)\|$ remains bounded for $T$ between $25$ and $800$, in agreement with Corollary~\ref{cor:multiplier}(ii).

\begin{figure}[ht]
\centering
\includegraphics[width=\textwidth]{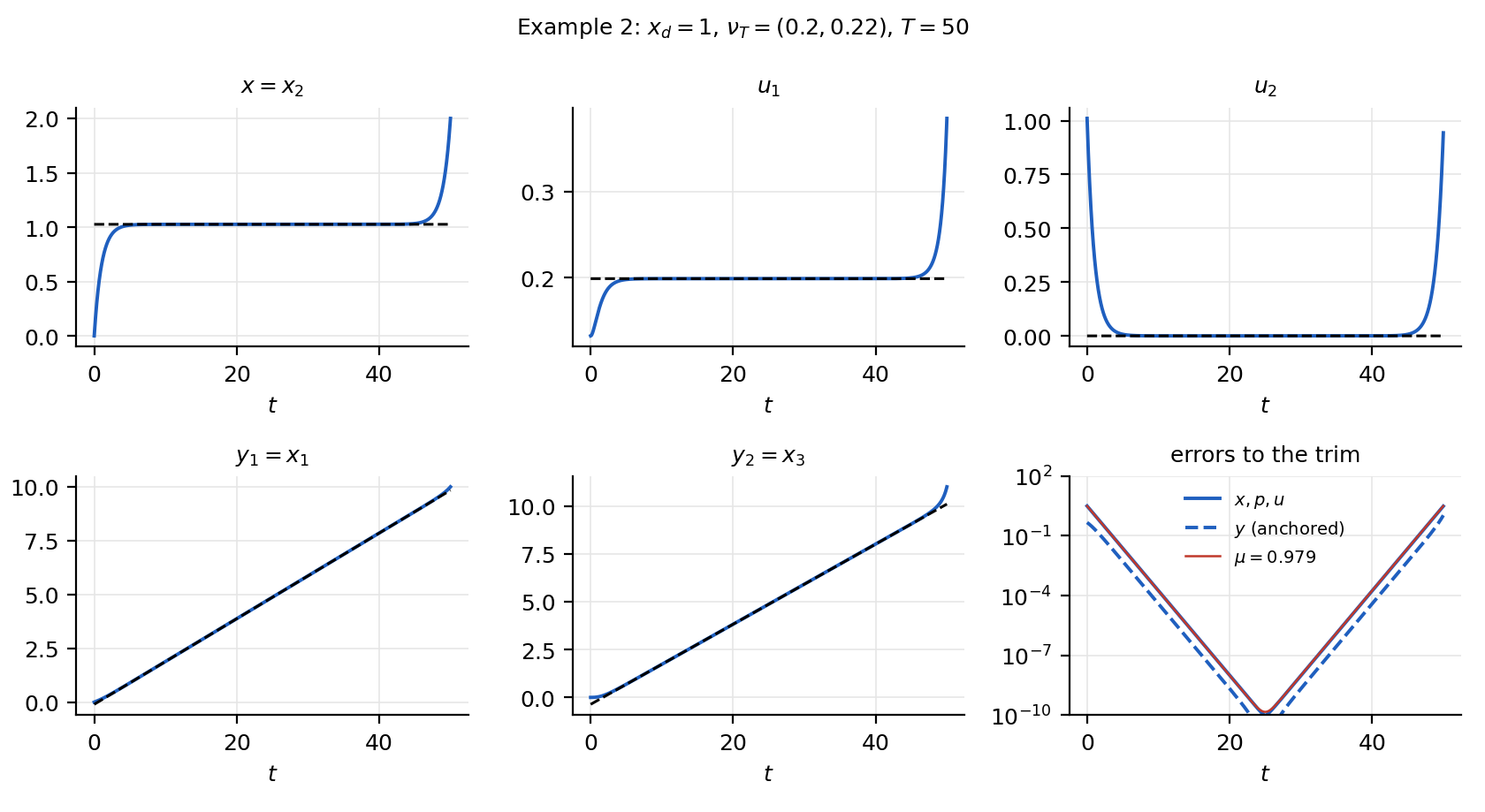}
\caption{Example 2 ($x_d=1$, $\nu_T=(0.2,0.22)$, $T=50$). Optimal solution (blue) and trim selected by $\lambda_T$ (dashed black), the cyclic components being anchored at $T/2$. Bottom right: distance to the trim in logarithmic scale, for $(x,p_x,u)$ and for $y$, with the reference rate $e^{-\mu t}+e^{-\mu(T-t)}$, $\mu\approx0.979$ (red).}
\label{fig:NLQ-1}
\end{figure}

\subsection{Example 3: Kepler orbital transfer problem}
\label{subsec:ex_kepler_full}

We now consider the Kepler example investigated in \cite{Flasskamp2025}, and reinterpret it in our framework where the cyclic variable is fixed at both endpoints. Here the cyclic variable is the angular coordinate $\theta$.
Let the \emph{shape} state and the \emph{cyclic} variables be
\[
x := (s,v_s,v_\theta)\in\mathbb{R}^3, \quad \text{and} \quad y := \theta\in\mathbb{R}.
\]
The dynamics of the controlled Kepler problem is
\begin{equation}
\begin{aligned}
\dot s &= v_s, \qquad
\dot v_s = s v_\theta^2 - \frac{k}{s^2} + u_s, \\
\dot \theta &= v_\theta, \qquad
\dot v_\theta = -\frac{2}{s}v_\theta v_s + \frac{1}{s^2}u_\theta,
\end{aligned}
\label{eq:Kepler-dyn}
\end{equation}
where $k>0$ is the gravitational constant (in suitable units) and $u=(u_s,u_\theta)\in\R^2$ is the thrust, the mass being normalized to $1$; in the simulations, $k=1016.895192894334$.
We consider the OCP given by
\begin{equation}\label{eq:Kepler-OCP}
\min_{u(\cdot)}\ J(u, x) = \frac{1}{2}\int_0^T \left( \|x-\tilde x \|^2_2 + \| u-\tilde u\|^2_2\right)\,dt
\quad \text{subject to } \eqref{eq:Kepler-dyn},
\end{equation}
with
\begin{equation}
\label{eq:Kepler-init}
\tilde x = \left(\tilde s,\ 0,\ \sqrt{\frac{k}{\tilde s^3}}\right) \quad \text{and} \quad \tilde u = (0,0)
\end{equation}
corresponding to a circular orbit of radius $\tilde s$, with $\tilde s=4.5$ in the simulations below.
In our boundary-constrained setting, we additionally impose endpoint conditions on the cyclic variable i.e.
\[
\theta(0)=\theta_0,\qquad \theta(T)=\theta_f,
\]
which are natural in transfer problems where the final angular phase is prescribed.

\medskip
\noindent\textit{Reduced BVP and trim reference.}
Since $\theta$ is cyclic, the associated adjoint
$p_\theta$ is constant: $p_\theta(t)\equiv \lambda\in\mathbb{R}$.

With the minimum convention of Section~\ref{sec:PMP_RBVP_ROCP} ($p^0=1$), the minimization condition gives $u_s=-p_{v_s}$ and $u_\theta=-p_{v_\theta}/s^2$, and the reduced Pontryagin boundary value problem associated with the reduced OCP reads
\begin{equation}
\label{eq:Kepler-RBVP}
\begin{aligned}
\dot s &= v_s, &&
\dot p_s = -(s-\tilde s)-p_{v_s}\Big(v_\theta^2+\frac{2{k}}{s^3}\Big)-\frac{2v_\theta v_s p_{v_\theta}}{s^2}-\frac{2p_{v_\theta}^2}{s^5}, \\
\dot v_s &= s v_\theta^2 - \frac{{k}}{s^2} - p_{v_s}, &&
\dot p_{v_s} = -(v_s-\tilde v_s)-p_s+\frac{2v_\theta p_{v_\theta}}{s}, \\
\dot v_\theta &= -\frac{2v_\theta v_s}{s} - \frac{p_{v_\theta}}{s^4}, &&
\dot p_{v_\theta} = -\lambda-2 s v_\theta p_{v_s}+\frac{2 v_s p_{v_\theta}}{s}-(v_\theta-\tilde v_\theta),
\end{aligned}
\end{equation}
where $\tilde v_s=0$ and $\tilde v_\theta={\sqrt{k/\tilde s^3}}$ are the components of $\tilde x$.
The static points of $\mathrm{ROCP}(\lambda)$ are the equilibria of this reduced Hamiltonian system, i.e.\ the points satisfying
\[
\begin{aligned}
&v_s^\circ = p_{v_\theta}^\circ = p_s^\circ = 0, \qquad p_{v_s}^\circ=s^\circ\,(v_\theta^\circ)^2-\frac{{k}}{(s^\circ)^2}, \\
&(s^\circ-\tilde s)+p_{v_s}^\circ\Big((v_\theta^\circ)^2+\frac{2{k}}{(s^\circ)^3}\Big)=0, \qquad
\lambda=-2s^\circ\,v_\theta^\circ\,p_{v_s}^\circ-(v_\theta^\circ-\tilde v_\theta).
\end{aligned}
\]

\emph{Verification of the assumptions.} The data $\theta_f=\pi T$ correspond to the mean cyclic velocity $\nu_T=\bar\nu=\pi$. For $\tilde s=4.5$, the static problem $\mathrm{SP}(\pi)$ has the solution $\bar v_\theta=\pi$, $\bar s\approx4.6878$, $\bar p_{v_s}\approx-6.34\cdot10^{-3}$ (so $\bar u_s\approx6.34\cdot10^{-3}$, $\bar u_\theta=0$), with cyclic multiplier $\bar\lambda\approx0.3858$; it is unique in $s\in[0.5,20]$. At this point, $\bar W$ has eigenvalues $0.930$, $1$ and $1.091$, so (H1) holds, and the Kalman matrix of the full linearization has rank $4=n+p$, so (H2) holds. The linearized reduced Hamiltonian system has eigenvalues $\pm0.192$ and $\pm0.615\pm3.125\,i$: the equilibrium is hyperbolic, and the exponential rate predicted by Lemma~\ref{lem:hyperbolic} is of order $0.19$.

For the numerical simulation, we take $T=100$, $(x_0,y_0) = (7,0,\sqrt{{k}/7^{3}},0)$ and $(x_f,y_f) = (3,0,\sqrt{{k}/3^{3}},\pi T)$. {We obtain $\lambda_T\approx0.4215$; the corresponding trim, computed as the equilibrium of \eqref{eq:Kepler-RBVP}, is $s^\circ_T\approx4.7048$, $v^\circ_{\theta,T}\approx3.1246$, $u^\circ_{s,T}\approx7.0\cdot10^{-3}$, $u^\circ_{\theta,T}=0$, and the linearization there has eigenvalues $\pm0.192$, $\pm0.615\pm3.107\,i$. The decay rate measured on the numerical solution is $0.191$, for the predicted rate $\mu=0.192$ (Figure~\ref{fig:Kepler-1}).
The multiplier satisfies $T|\lambda_T-\bar\lambda|\approx4.07$, $3.57$, $3.36$ for $T=50$, $100$, $200$ (Figure~\ref{fig:lambda}, right). The boundary radii $7$ and $3$ are not close to $\bar s$: these data lie outside the smallness condition \eqref{eq:smallness}, but the observed behavior is the one predicted by Theorem~\ref{thm:exp_trim_turnpike}.}

\begin{figure}[ht]
\centering
\includegraphics[width=\textwidth]{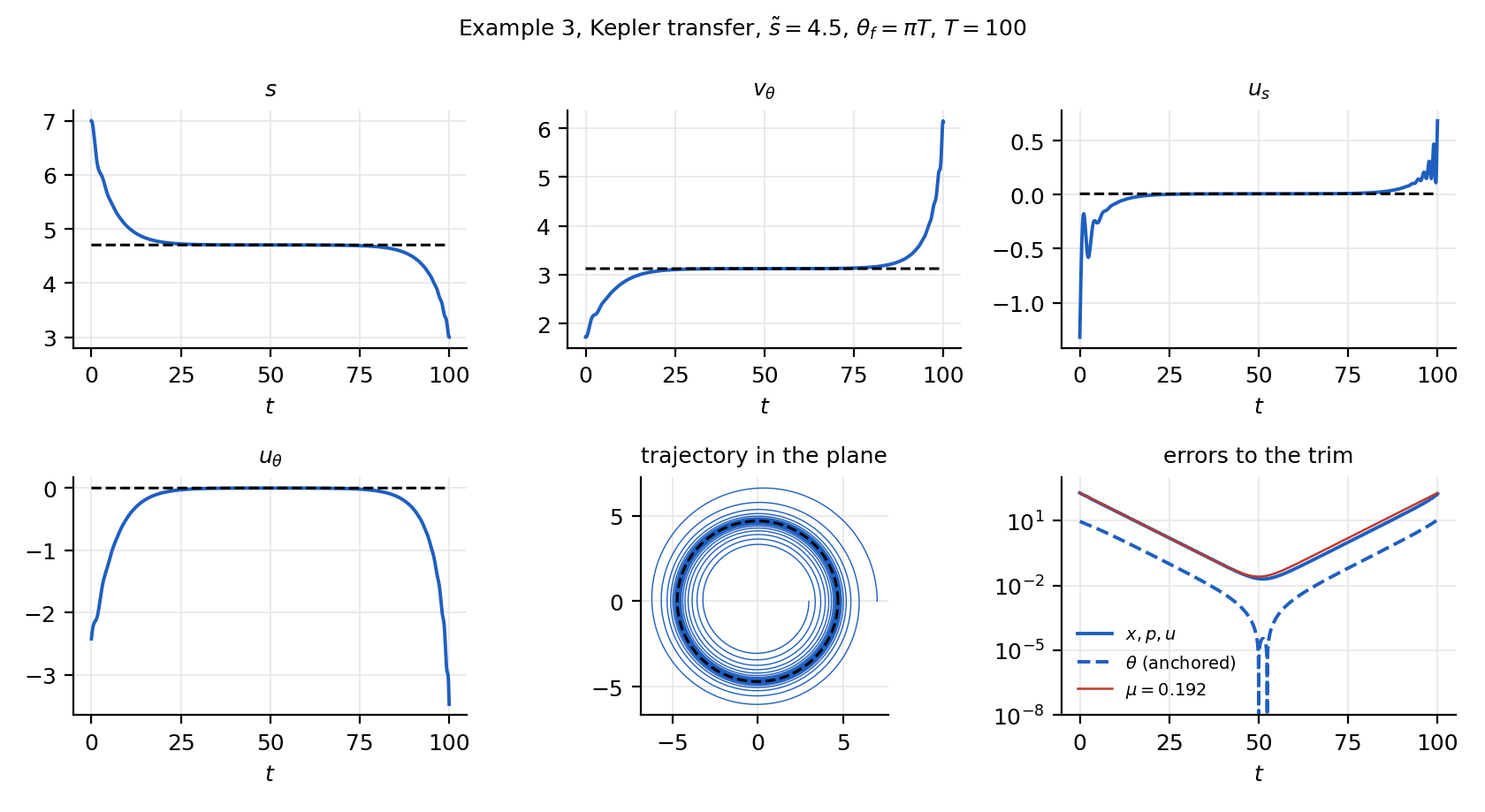}
\vspace{-0.5cm}
\caption{{Example 3 ($\tilde s=4.5$, $\theta_f=\pi T$, $T=100$). Optimal solution (blue) and trim selected by $\lambda_T$ (dashed black), the angle being anchored at $T/2$. Bottom right: distance to the trim in logarithmic scale for $(x,p_x,u)$ and for $\theta$, with the reference rate $\mu\approx0.192$ (red).}}
\label{fig:Kepler-1}
\end{figure}

\begin{figure}[ht]
\centering
\includegraphics[width=0.48\textwidth]{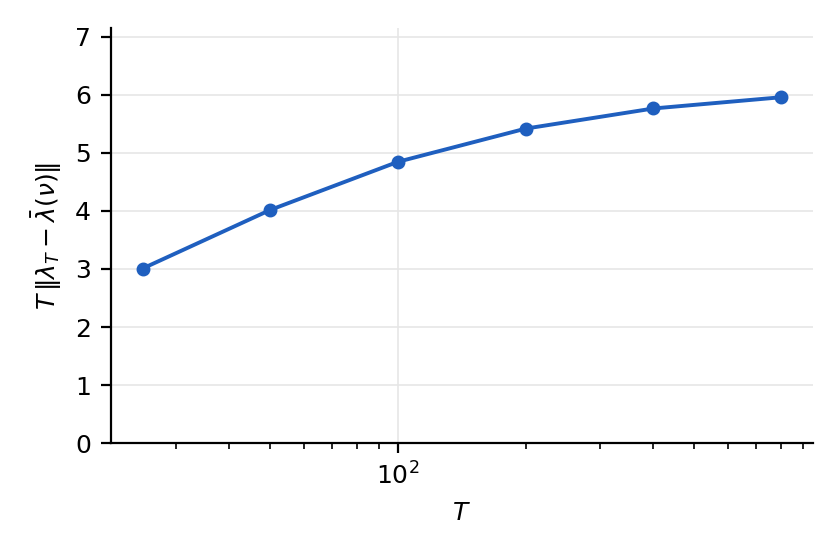}\hfill
\includegraphics[width=0.48\textwidth]{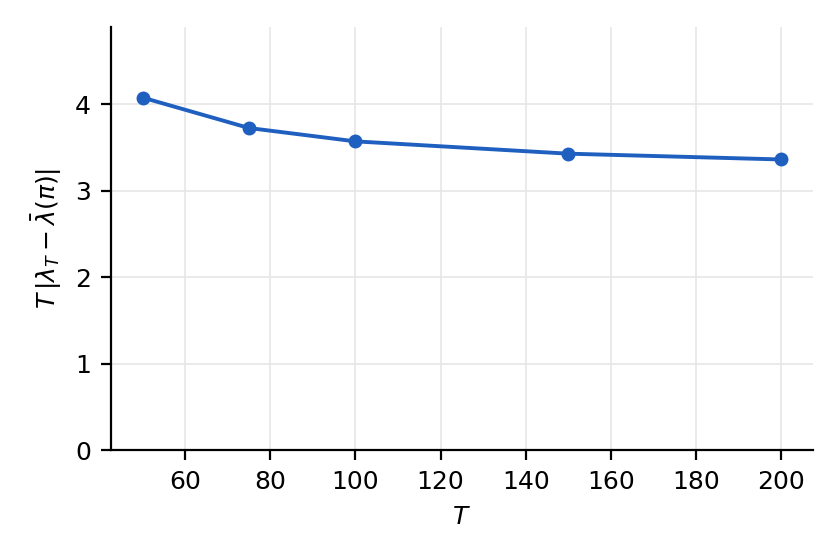}
\vspace{-0.3cm}
\caption{Cyclic multiplier versus the multiplier of the static problem with prescribed velocity. Left, Example~2: $T\,\|\lambda_T-(v^\circ)^{-1}(\nu_T)\|$ as a function of $T$ (logarithmic scale). Right, Example~3: $T\,|\lambda_T-\bar\lambda|$ for $T\in\{50,75,100,150,200\}$. Both remain bounded, i.e.\ $\lambda_T-(v^\circ)^{-1}(\nu_T)=O(1/T)$.}
\label{fig:lambda}
\label{fig:NLQ-2}
\label{fig:Kepler-2}
\end{figure}

\section{Conclusion}
\label{sec:conclusion}

We have established an exponential trim turnpike property for optimal control problems with cyclic variables and fixed endpoint conditions on all state variables. The cyclic adjoint is a constant multiplier $\lambda_T$, which reduces the optimality system to that of a reduced problem $\mathrm{ROCP}(\lambda_T)$. The main challenge lies in the fact that the fixed final condition on the cyclic variable leads to a cyclic multiplier which depends on the time horizon, and so do the corresponding reduced problem and its spectral gap. To overcome this difficulty, we make use of the reference trim, defined from the limit of the mean cyclic velocity. Under local assumptions on a reference trim (positivity of a reduced Hessian and a Kalman condition for the full linearization), we proved that OCPs with boundary data of the shape variables close to those of the reference trim admit the exponential turnpike property with respect to the trim parametrized by $\lambda_T$ and anchored at the midpoint.  In particular, the derived exponential estimates are uniform in $T$.
Comparing the obtained results with the linear turnpike theorem \cite{Tre:23}, we found that the difference in the turnpike estimates, linear in the case of \cite{Tre:23} and exponential in our case, comes from the identification of the trim turnpike trajectory:
the multiplier of the static problem with prescribed mean cyclic velocity differs from $\lambda_T$ by $O(1/T)$, which produces the $1/T$ term and the merely bounded error on the cyclic variables. When the final values of the cyclic variables are free, the result reduces to \cite{Flasskamp2025}. The examples confirm the predicted rates and show that the trim turnpike is non-static, while the reference trim is non-static only when the prescribed cyclic displacement grows linearly with the horizon.

Several questions remain open, for instance, a global version of the result, under dissipativity-type assumptions adapted to the trim \cite{Tre:23, TrelatZhang2018}, which would apply to global minimizers; the extension to problems without shape-cyclic coordinates, in particular with non-Abelian symmetry groups, where the cyclic adjoint is replaced by a momentum map; and the analysis of structure-preserving time discretizations, for which the uniform estimates obtained here are a natural starting point.

\bibliography{Biblio}

\begin{appendices}

\section{Technical proofs}
\label{app:proofs}

Recall that $H^\lambda(x,p,u)$ denotes the Hamiltonian \eqref{eq:H_ROCP} of $\mathrm{ROCP}(\lambda)$, and set $\bar H:=H^{\bar\lambda}(\cdot,\bar p_x,\cdot)$. Since $(\bar x,\bar u,\bar p_x)$ is a static point of $\mathrm{ROCP}(\bar\lambda)$, its KKT conditions read
\begin{equation}
\nabla_{(x,u)}\bar H(\bar x,\bar u)=0,\qquad f_1(\bar x)+F_2(\bar x)\bar u=0 .
\label{eq:KKT_static}
\end{equation}
We set
\[
E_1:=\nabla^2_{(x,u)}\bar H(\bar x,\bar u)=\begin{pmatrix}\bar H_{xx}&\bar H_{xu}\\ \bar H_{ux}&R\end{pmatrix},\qquad E_2:=(A\ \ B),\qquad J_2:=(C\ \ D),
\]
so that $J=\begin{pmatrix}E_2\\ J_2\end{pmatrix}$. By (H1), $E_1$ is positive definite. Indeed, for $(\xi,\omega)\in\R^n\times\R^m$,
\[
(\xi,\omega)^\top E_1(\xi,\omega)=\xi^\top\bar W\xi+\big(\omega+R^{-1}\bar H_{ux}\xi\big)^\top R\big(\omega+R^{-1}\bar H_{ux}\xi\big),
\]
and both terms are nonnegative and vanish together only if $\xi=0$ and $\omega=0$.

\subsection{Proof of Lemma~\ref{lem:trim_branch}}
\label{app:trim_branch}

\emph{(a)} The KKT system of \eqref{eq:static_ROCP} reads $\Phi(x,u,p;\lambda)=0$, where
\[
\Phi(x,u,p;\lambda):=\Big(\nabla_xH^\lambda(x,p,u),\ \nabla_uH^\lambda(x,p,u),\ f_1(x)+F_2(x)u\Big),
\]
and $\Phi(\bar x,\bar u,\bar p_x;\bar\lambda)=0$ by \eqref{eq:KKT_static}. The Jacobian of $\Phi$ with respect to $(x,u,p)$ at this point is
\[
\begin{pmatrix}E_1&E_2^\top\\ E_2&0\end{pmatrix}.
\]
By Lemma~\ref{Lem.H2}, $(A,B)$ satisfies the Kalman condition, and the Hautus test with $s=0$ gives $\operatorname{rank}(A\ \ B)=n$, i.e.\ $E_2$ is surjective. This Jacobian is therefore invertible.
Indeed, if $E_1Y_1+E_2^\top Y_2=0$ and $E_2Y_1=0$, then $Y_1^\top E_1Y_1=-(E_2Y_1)^\top Y_2=0$, so $Y_1=0$, and then $Y_2=0$ since $E_2^\top$ is injective.

Thus, by the implicit function theorem, there exist $\rho_1>0$, a neighborhood $\mathcal V$ of $(\bar x,\bar u,\bar p_x)$ and a $C^1$ map $\lambda\mapsto(x^\circ,u^\circ,p_x^\circ)(\lambda)$ such that, for $\|\lambda-\bar\lambda\|\leq\rho_1$, $(x^\circ,u^\circ,p_x^\circ)(\lambda)$ is the only solution of $\Phi(\cdot;\lambda)=0$ in $\mathcal V$, which is the local uniqueness stated in (a).

\emph{(b)} Since $R\succ0$, $H^\lambda$ is strictly convex in $u$ and $\nabla_uH^\lambda(x,p,u)=0$ if and only if $u=u^\star(x,p,\lambda)$.

Now, let $h_\lambda(x,p):=H^\lambda(x,p,u^\star(x,p,\lambda))$. Since $\nabla_uH^\lambda=0$ at $u=u^\star$, we have $X_\lambda=(\nabla_ph_\lambda,-\nabla_xh_\lambda)=\big(\nabla_pH^\lambda,-\nabla_xH^\lambda\big)(x,p,u^\star)$. At a solution of $\Phi=0$, $u=u^\star$, $\nabla_pH^\lambda=f_1+F_2u=0$ and $\nabla_xH^\lambda=0$, so that $z^\circ(\lambda)$ is an equilibrium of $X_\lambda$ and $u^\circ(\lambda)=u^\star(z^\circ(\lambda),\lambda)$.

Denote by $A(\lambda)$, $B(\lambda)$, $H_{xx}(\lambda)$, $H_{xu}(\lambda)$ the matrices $D_x(f_1+F_2u)$, $F_2$, $\nabla^2_{xx}H^\lambda$, $\nabla^2_{xu}H^\lambda$ evaluated at $(x^\circ(\lambda),p_x^\circ(\lambda),u^\circ(\lambda))$. Differentiating $Ru^\star+F_2^\top p+G_2^\top\lambda=0$ gives $\partial_xu^\star=-R^{-1}H_{ux}(\lambda)$ and $\partial_pu^\star=-R^{-1}B(\lambda)^\top$ at the trim, and a direct computation gives
\begin{equation}
M(\lambda)=\begin{pmatrix}\tilde A(\lambda)&-S(\lambda)\\ -W(\lambda)&-\tilde A(\lambda)^\top\end{pmatrix},
\label{eq:M_lambda}
\end{equation}
with
$
\tilde A(\lambda)=A(\lambda)-B(\lambda)R^{-1}H_{ux}(\lambda)$, $S(\lambda) = B(\lambda)R^{-1} B(\lambda)^\top$, and $W(\lambda)=H_{xx}(\lambda)-H_{xu}(\lambda)R^{-1}H_{ux}(\lambda).$

At $\lambda=\bar\lambda$, $W(\bar\lambda)=\bar W\succ0$ by (H1), and $(\tilde A(\bar\lambda),B)$ satisfies the Kalman condition, since $(A,B)$ does and $[sI-\tilde A,\ B]=[sI-A,\ B]\begin{pmatrix}I&0\\ R^{-1}\bar H_{ux}&I\end{pmatrix}$ for every $s\in\C$. Both properties are open and the matrices depend continuously on $\lambda$, so that, for some $\rho_\Lambda\in(0,\rho_1]$,
\begin{equation}
W(\lambda)\succ0 ~~ \text{and} ~~ (\tilde A(\lambda),B(\lambda))\ \text{satisfies the Kalman condition, for every }\lambda\in\Lambda.
\label{eq:M_lambda_prop}
\end{equation}

For $\lambda=\bar\lambda$, the hyperbolicity of $M(\bar\lambda)$ is the content of Lemma~\ref{lem:hyp_ref}. The following self-contained argument, which only uses \eqref{eq:M_lambda_prop}, shows that it holds for every $\lambda\in\Lambda$.

Now, let $\lambda\in\Lambda$, and assume that $M(\lambda)(\xi,\eta)=s(\xi,\eta)$ with $s\in i\R$ and $(\xi,\eta)\in\C^{2n}\setminus\{0\}$, i.e.\ $\tilde A\xi-S\eta=s\xi$ and $-W\xi-\tilde A^\top\eta=s\eta$. Multiplying the first equation by $\eta^*$ on the left, the conjugate transpose of the second one by $\xi$ on the right, and adding, we get $-\eta^*S\eta-\xi^*W\xi=(s+\bar s)\,\eta^*\xi=0$. Since $W\succ0$ and $S=BR^{-1}B^\top$ with $R\succ0$, this gives $\xi=0$ and $B^\top\eta=0$.
With $\xi=0$, the second equation $W\xi+\tilde A^\top\eta=-s\eta$ reduces to $\tilde A^\top\eta=-s\eta$, i.e.\ $\eta^*\tilde A=-\bar s\,\eta^*$, and since $\eta^*B=0$, the Hautus test for $(\tilde A,B)$ gives $\eta=0$, a contradiction. Hence $M(\lambda)$ has no eigenvalue on the imaginary axis. \hfill$\square$

\subsection{Proof of Lemma~\ref{lem:hyperbolic}}
\label{app:hyperbolic}

We treat the nonlinear part of $X$ as a perturbation of its linearization at $z^\circ(\lambda)$; the only point requiring care is that all constants must be uniform in $\lambda\in\Lambda$ (see \cite{GruneSchallerSchiela2020} for estimates of this type in the sensitivity approach to turnpike).

Let $M(\lambda):=D_zX(z^\circ(\lambda),\lambda)$, which is continuous in $\lambda$. Since $M(\lambda)$ is hyperbolic, i.e.\ $\sigma(M(\lambda))\cap i\R=\emptyset$, the compactness of $\Lambda$ and the continuity of the spectrum \cite[Chapter~II]{Kato1995} give $\mu>0$ such that $|\operatorname{Re}s|\geq3\mu$ for every eigenvalue $s$ of $M(\lambda)$ and every $\lambda\in\Lambda$. Let $P_s(\lambda)$ and $P_u(\lambda)=I-P_s(\lambda)$ be the spectral projections onto the stable and unstable subspaces of $M(\lambda)$. Then $M(\lambda)$ admits an exponential dichotomy \cite{Coppel1978}: there is $K_0\geq1$ such that
\begin{equation}
\|e^{M(\lambda)t}P_s(\lambda)\|\leq K_0e^{-2\mu t},\qquad \|e^{-M(\lambda)t}P_u(\lambda)\|\leq K_0e^{-2\mu t},\qquad t\geq0,\ \lambda\in\Lambda .
\label{eq:app_dichotomy}
\end{equation}
The constant $K_0$ can be taken independent of $\lambda$. Indeed,
\[
e^{M(\lambda)t}P_s(\lambda)=\frac1{2\pi i}\oint_{\Gamma}e^{zt}(zI-M(\lambda))^{-1}dz,
\]
where $\Gamma$ is the boundary of the fixed rectangle $Q=\{z\in\C:-R_0\leq\operatorname{Re}z\leq-2\mu,\ |\operatorname{Im}z|\leq R_0\}$, with $R_0>\max_\Lambda\|M(\lambda)\|$, which contains the stable spectrum of every $M(\lambda)$ \cite[Chapter~I, \S5]{Kato1995}. Hence $K_0$ only depends on the length of $\Gamma$ and on a bound of the resolvent on the compact set $\Gamma\times\Lambda$; the unstable part is treated in the same way. Finally, since $X$ is of class $C^2$ and $X(z^\circ(\lambda),\lambda)=0$, Taylor's formula on the compact set $\{(z^\circ(\lambda)+w,\lambda):\|w\|\leq r_0,\ \lambda\in\Lambda\}$, for a fixed $r_0>0$ and with $L:=\frac12\max\|D^2_zX\|$ on this set, gives that $N(w,\lambda):=X(z^\circ(\lambda)+w,\lambda)-M(\lambda)w$ satisfies $\|N(w,\lambda)\|\leq L\|w\|^2$ for $\|w\|\leq r_0$ and $\lambda\in\Lambda$.

Fix $\lambda\in\Lambda$, $T>0$ and a solution $z$ with $\|z(t)-z^\circ(\lambda)\|\leq r$ on $[0,T]$, where $r\leq r_0$ is chosen below. Set $w=z-z^\circ(\lambda)$, $N(t)=N(w(t),\lambda)$, $a=\|w(0)\|$, $b=\|w(T)\|$, and drop the dependence on $\lambda$. Since $P_s$ and $P_u$ commute with $M$, integrating the stable (resp. unstable) component of $\dot w=Mw+N(t)$ forward from $0$ (resp. backward from $T$) gives
\[
\begin{aligned}w(t)={}&e^{Mt}P_sw(0)+\int_0^te^{M(t-s)}P_sN(s)\,ds\\&+e^{-M(T-t)}P_uw(T)-\int_t^Te^{-M(s-t)}P_uN(s)\,ds,\end{aligned}
\]
where every exponential is evaluated at a nonnegative time, so that \eqref{eq:app_dichotomy} applies to each term and gives
\[
\|w(t)\|\leq K_0\big(ae^{-\mu t}+be^{-\mu(T-t)}\big)+K_0\int_0^Te^{-2\mu|t-s|}\|N(s)\|\,ds .
\]
For $\delta>0$, let $\psi(t):=ae^{-\mu t}+be^{-\mu(T-t)}+\delta$ and $m:=\sup_{[0,T]}\|w\|/\psi$, which is finite since $\psi\geq\delta$. We have $\|N(s)\|\leq Lr\|w(s)\|\leq Lrm\,\psi(s)$, and direct integration of each term of $\psi$ gives
\[
\int_0^Te^{-2\mu|t-s|}\psi(s)\,ds\leq\frac{2\psi(t)}\mu,\qquad t\in[0,T].
\]
Hence $\|w(t)\|\leq K_0\psi(t)+\frac{2K_0Lrm}{\mu}\psi(t)$, that is, $m\leq K_0+\frac{2K_0Lr}{\mu}m$. Choosing $r=\min\{r_0,\mu/(4K_0L)\}$ yields $m\leq2K_0$, and letting $\delta\to0$ gives \eqref{eq:hyperbolic_estimate} with $K=2K_0$. The constants $r$, $K$, $\mu$ only depend on $\Lambda$, $X$ and $z^\circ$. If $X(\cdot,\lambda)$ is affine, then $L=0$ and the restriction $\|z-z^\circ(\lambda)\|\leq r$ is not needed, which is used in Remark~\ref{rem:LQ_global}. \hfill$\square$

\subsection{Proof of Lemma~\ref{lem:trim_velocity}}
\label{app:trim_branch_2}

\emph{Value at $\bar\lambda$.} Since $(x^\circ,u^\circ)(\bar\lambda)=(\bar x,\bar u)$, the identity $v^\circ(\bar\lambda)=\bar\nu$ is the definition \eqref{eq:nu_bar}.

\emph{Derivative of the branch.} Let $\delta\lambda\in\R^p$ and set $w:=D(x^\circ,u^\circ)(\bar\lambda)\,\delta\lambda\in\R^{n+m}$ and $\delta p:=Dp_x^\circ(\bar\lambda)\,\delta\lambda\in\R^n$. We differentiate the identity $\Phi\big((x^\circ,u^\circ,p_x^\circ)(\lambda);\lambda\big)=0$, valid on $\Lambda$ by Lemma~\ref{lem:trim_branch}(a), at $\lambda=\bar\lambda$ in the direction $\delta\lambda$. The derivative of $\Phi$ with respect to $(x,u,p)$ is the KKT matrix computed in the proof of Lemma~\ref{lem:trim_branch}(a). The parameter $\lambda$ only enters $\Phi$ through the term $\lambda^\top(g_1(x)+G_2(x)u)$ of $H^\lambda$, whose gradient with respect to $(x,u)$ at $(\bar x,\bar u)$ is $J_2^\top\lambda$, while the constraint $f_1(x)+F_2(x)u$ does not depend on $\lambda$. Hence $\partial_\lambda\Phi=\big(J_2^\top,\,0\big)$ at the reference point, and we obtain the linear system
\begin{equation}
E_1w+E_2^\top\delta p+J_2^\top\delta\lambda=0,\qquad E_2w=0 .
\label{eq:app_dbranch}
\end{equation}
Moreover, since $v^\circ(\lambda) =(g_1+G_2u) \big(x^\circ(\lambda), u^\circ(\lambda)\big)$, the chain rule gives $Dv^\circ(\bar\lambda) \, \delta\lambda = J_2w$.

\emph{Invertibility of $Dv^\circ(\bar\lambda)$.} Since $Dv^\circ(\bar\lambda)$ is a $p\times p$ matrix, it suffices to show that it is injective. Assume that $J_2w=0$. Multiplying the first equation of \eqref{eq:app_dbranch} by $w^\top$ and using the second one, we get $w^\top E_1w =-(\underbrace{E_2w}_{=0})^\top\delta p-(\underbrace{J_2w}_{=0})^\top\delta\lambda=0$.

Since $E_1\succ0$ by (H1), this gives $w=0$: the trim does not move at first order. The first equation of \eqref{eq:app_dbranch} then reduces to $E_2^\top\delta p+J_2^\top\delta\lambda=J^\top(\delta p,\delta\lambda)=0$. Since $J$ is surjective by (H2) and Lemma~\ref{Lem.H2}, $J^\top$ is injective, so that $\delta p=0$ and $\delta\lambda=0$. Thus $Dv^\circ(\bar\lambda)$ is invertible. (In fact, if $Q$ denotes the upper left block of the inverse of the KKT matrix, \eqref{eq:app_dbranch} gives $w=-QJ_2^\top\delta\lambda$, hence $Dv^\circ(\bar\lambda)=-J_2QJ_2^\top$. Since $Q=E_1^{-1/2}(I-\Pi)E_1^{-1/2}\succeq0$, where $\Pi$ is the orthogonal projector onto the range of $E_1^{-1/2}E_2^\top$, and $Dv^\circ(\bar\lambda)$ is injective, $Dv^\circ(\bar\lambda)$ is symmetric negative definite.)

\emph{Quantitative inverse.} The map $v^\circ$ is of class $C^1$ on $\Lambda$, as the composition of the $C^1$ branch of Lemma~\ref{lem:trim_branch} with $g_1+G_2u$. Let $D_0:=Dv^\circ(\bar\lambda)$ and $\kappa:=1/(2\|D_0^{-1}\|)$. By continuity of $Dv^\circ$, after shrinking $\rho_\Lambda$, we may assume that $\|Dv^\circ(\lambda)-D_0\|\leq\kappa$ for every $\lambda$ in the (convex) ball $\Lambda$.

(a) Let $\lambda,\lambda'\in\Lambda$. Since the segment $[\lambda',\lambda]$ lies in $\Lambda$,
\[
v^\circ(\lambda)-v^\circ(\lambda')-D_0(\lambda-\lambda')=\int_0^1\Big(Dv^\circ\big(\lambda'+s(\lambda-\lambda')\big)-D_0\Big)ds\,(\lambda-\lambda'),
\]
whose norm is at most $\kappa\|\lambda-\lambda'\|$.

On the other hand, $\|\lambda-\lambda'\|\leq\|D_0^{-1}\|\,\|D_0(\lambda-\lambda')\|$, i.e.\ $\|D_0(\lambda-\lambda')\|\geq2\kappa\|\lambda-\lambda'\|$. By the reverse triangle inequality, $\|v^\circ(\lambda)-v^\circ(\lambda')\|\geq2\kappa\|\lambda-\lambda'\|-\kappa\|\lambda-\lambda'\|=\kappa\|\lambda-\lambda'\|$. In particular, $v^\circ$ is injective on $\Lambda$ and its inverse is Lipschitz with constant $\kappa^{-1}$ on $v^\circ(\Lambda)$.

(b) Let $\bar B(\lambda',\sigma)\subset\Lambda$ and $\nu\in\R^p$ with $\|\nu-v^\circ(\lambda')\|\leq\kappa\sigma$. Consider the map
\[
\Theta(\lambda):=\lambda+D_0^{-1}\big(\nu-v^\circ(\lambda)\big),\qquad\lambda\in\bar B(\lambda',\sigma),
\]
whose fixed points are exactly the solutions of $v^\circ(\lambda)=\nu$. For $\lambda,\lambda''\in\bar B(\lambda',\sigma)$, $\Theta(\lambda)-\Theta(\lambda'')=D_0^{-1}\left[D_0(\lambda-\lambda'')-\big(v^\circ(\lambda)-v^\circ(\lambda'')\big)\right]$, so that, by the estimate in (a), $\|\Theta(\lambda)-\Theta(\lambda'')\|\leq\|D_0^{-1}\|\kappa\|\lambda-\lambda''\|=\frac12\|\lambda-\lambda''\|$. Moreover, $\Theta$ maps $\bar B(\lambda',\sigma)$ into itself, since
\[
\begin{aligned}
\|\Theta(\lambda)-\lambda'\| &\leq \|\Theta(\lambda)-\Theta(\lambda')\|+\|\Theta(\lambda')-\lambda'\| \\
& \leq\tfrac12\sigma + \|D_0^{-1}\|\, \|\nu-v^\circ(\lambda')\| \leq \tfrac12\sigma+\|D_0^{-1}\|\kappa\sigma=\sigma .
\end{aligned}
\]
By the Banach fixed point theorem, $\Theta$ has a fixed point $\lambda\in\bar B(\lambda',\sigma)$, i.e.\ $v^\circ(\lambda)=\nu$. Note that the radius $\kappa\sigma$ does not depend on the base point $\lambda'$; this uniformity is used in the proofs of Proposition~\ref{prop:existence} and Corollary~\ref{cor:multiplier}. \hfill$\square$

\subsection{Proof of Proposition~\ref{prop:existence}}
\label{app:existence}

The proof is structured as follows. For each fixed $\lambda$ close to $\bar\lambda$, we first solve the reduced BVP~\eqref{eq:RBVP} near the hyperbolic equilibrium $z^\circ(\lambda)$ by a contraction argument (Steps~1--2). We then select $\lambda=\lambda_T$ so that the cyclic condition $y(T)=y_f$ holds, which yields the extremal of item~(i) (Steps~3--4). Finally, the strict convexity of the Hamiltonian in $(x,u)$ near the reference trim gives local optimality and uniqueness (Steps~5--6).

All constants below depend only on the data of \eqref{eq:OCP} and on the reference trim. We use $M(\lambda)$, $P_s(\lambda)$, $P_u(\lambda)$, $\mu$, $K_0$, $N$, $r_0$, $L$ from Appendix~\ref{app:hyperbolic}; since $\|D_z^2X\|\leq2L$, we have $\|N(w,\lambda)-N(w',\lambda)\|\leq2L\max\{\|w\|,\|w'\|\}\|w-w'\|$ for $\|w\|,\|w'\|\leq r_0$. The constants $\ell$, $r$, $K$ are fixed as after Lemma~\ref{lem:hyperbolic}, and $\ell_u$, $L_v$ are Lipschitz constants of $u^\star$ and $v$ on $\mathcal O:=\{(z,\lambda):\lambda\in\Lambda,\ \|z-z^\circ(\lambda)\|\leq r\}$. We write $z=(x,p_x)$, $\bar z=(\bar x,\bar p_x)$, and $\pi_x(\xi,\eta):=\xi$.

\emph{Step 1 (linear two-point problem).} Let $\lambda\in\Lambda$, $M=M(\lambda)$, and let $E_s$, $E_u$ be its stable and unstable subspaces. We first show that $E_s\cap(\{0\}\times\R^n)=\{0\}$. Let $(0,\eta_0)\in E_s$ and $(\xi(t),\eta(t)):=e^{Mt}(0,\eta_0)$, which tends to $0$ as $t\to+\infty$. By \eqref{eq:M_lambda},
\[
\frac{d}{dt}\big(\xi^\top\eta\big)=(\tilde A\xi-S\eta)^\top\eta+\xi^\top(-W\xi-\tilde A^\top\eta)=-\eta^\top S\eta-\xi^\top W\xi .
\]
Integrating on $[0,+\infty)$, with $\xi(0)=0$, gives $\int_0^\infty(\eta^\top S\eta+\xi^\top W\xi)\,dt=0$. By \eqref{eq:M_lambda_prop}, $W\succ0$, hence $\xi\equiv0$ and $B^\top\eta\equiv0$, with $\eta(t)=e^{-\tilde A^\top t}\eta_0$. Differentiating $\eta_0^\top e^{-\tilde At}B\equiv0$ at $t=0$ gives $\eta_0^\top\tilde A^kB=0$ for all $k$, so that $\eta_0=0$ by the Kalman condition in \eqref{eq:M_lambda_prop}. The same argument, with $t\to-\infty$, applies to $E_u$. Hence $\pi_x$ is injective on $E_s$ and on $E_u$, both subspaces have dimension $n$, and
\[
E_s=\{(\alpha,\Sigma_s\alpha):\alpha\in\R^n\},\qquad E_u=\{(\beta,\Sigma_u\beta):\beta\in\R^n\},
\]
for matrices $\Sigma_s(\lambda)$, $\Sigma_u(\lambda)$ which depend continuously on $\lambda$, as $P_s(\lambda)$ and $P_u(\lambda)$ do. Let $\sigma_0:=\max_\Lambda\max\{\|\Sigma_s\|,\|\Sigma_u\|\}$ and let $T_2\geq1$ be such that $K_0(1+\sigma_0)e^{-2\mu T_2}\leq\frac12$.

We now show that there is $K_1>0$ such that, for every $\lambda\in\Lambda$, $T\geq T_2$, $a,b\in\R^n$ and $g\in C([0,T];\R^{2n})$, the problem
\begin{equation}
\dot w=M(\lambda)w+g(t),\qquad \pi_xw(0)=a,\quad \pi_xw(T)=b,
\label{eq:app_linear_bvp}
\end{equation}
has a unique solution $w=:\mathcal S_\lambda(a,b,g)$, which satisfies $\|w\|_\infty\leq K_1\big(\|a\|+\|b\|+\|g\|_\infty\big)$. Let $q(t):=\int_0^te^{M(t-s)}P_sg(s)\,ds-\int_t^Te^{-M(s-t)}P_ug(s)\,ds$; it solves $\dot q=Mq+g$, and $\|q\|_\infty\leq(K_0/\mu)\|g\|_\infty$ by \eqref{eq:app_dichotomy}. Every solution of $\dot w=Mw+g$ can be written in a unique way as
\[
w(t)=e^{Mt}(\alpha,\Sigma_s\alpha)+e^{-M(T-t)}(\beta,\Sigma_u\beta)+q(t),\qquad\alpha,\beta\in\R^n,
\]
where the first term decays forward from $t=0$ and the second one backward from $t=T$. The boundary conditions in \eqref{eq:app_linear_bvp} read
\[
\alpha+\pi_xe^{-MT}(\beta,\Sigma_u\beta)=a-\pi_xq(0),\qquad \pi_xe^{MT}(\alpha,\Sigma_s\alpha)+\beta=b-\pi_xq(T).
\]
By \eqref{eq:app_dichotomy} and the choice of $T_2$, the coupling terms are bounded by $\frac12\|\beta\|$ and $\frac12\|\alpha\|$. This linear system in $(\alpha,\beta)$ is therefore uniquely solvable, with $\|\alpha\|+\|\beta\|\leq2\big(\|a\|+\|b\|+2\|q\|_\infty\big)$, and the claim follows with $K_1:=2K_0(1+\sigma_0)(1+2K_0/\mu)+K_0/\mu$. For fixed $T$, $\mathcal S_\lambda$ depends continuously on $\lambda$, uniformly on bounded sets of data $(a,b,g)$.

\emph{Step 2 (reduced boundary value problem for fixed $\lambda$).} Let $r_1:=\min\{r_0,r,1/(8K_1L)\}$ (with $1/0:=+\infty$), set $a(\lambda):=x_0-x^\circ(\lambda)$, $b(\lambda):=x_f-x^\circ(\lambda)$, and assume $T\geq T_2$ and $\|a(\lambda)\|+\|b(\lambda)\|\leq r_1/(2K_1)$. Writing $z=z^\circ(\lambda)+w$, the problem \eqref{eq:RBVP} becomes $\dot w=M(\lambda)w+N(w,\lambda)$ with $\pi_xw(0)=a(\lambda)$ and $\pi_xw(T)=b(\lambda)$, that is, by Step~1, the fixed point problem
\[
w=\mathcal T_\lambda(w):=\mathcal S_\lambda\big(a(\lambda),b(\lambda),N(w(\cdot),\lambda)\big).
\]
On the ball $\{\|w\|_\infty\leq r_1\}$, Step~1 and the bounds on $N$ give $\|\mathcal T_\lambda(w)\|_\infty\leq K_1(\|a\|+\|b\|)+K_1Lr_1^2\leq r_1$ and $\|\mathcal T_\lambda(w)-\mathcal T_\lambda(w')\|_\infty\leq\frac12\|w-w'\|_\infty$. Hence $\mathcal T_\lambda$ has a unique fixed point $w_\lambda$ in this ball and, since $K_1L\|w_\lambda\|_\infty\leq K_1Lr_1\leq\frac18$, $\|w_\lambda\|_\infty\leq K_1(\|a\|+\|b\|)+\frac18\|w_\lambda\|_\infty$, so that
\begin{equation}
\|w_\lambda\|_\infty\leq2K_1\big(\|a(\lambda)\|+\|b(\lambda)\|\big).
\label{eq:app_w_lambda}
\end{equation}
Since $\mathcal T_\lambda$ is a contraction depending continuously on $\lambda$, $\lambda\mapsto w_\lambda\in C([0,T];\R^{2n})$ is continuous. Finally, $\|w_\lambda\|_\infty\leq r$, so that Lemma~\ref{lem:hyperbolic} and \eqref{eq:app_w_lambda} give
\begin{equation}
\|w_\lambda(t)\|\leq2KK_1\big(\|a(\lambda)\|+\|b(\lambda)\|\big)\left(e^{-\mu t}+e^{-\mu(T-t)}\right),\qquad t\in[0,T].
\label{eq:app_w_exp}
\end{equation}

\emph{Step 3 (choice of the cyclic multiplier).} Let $\sigma:=\min\{\rho_\Lambda/2,\ r_1/(8K_1\ell)\}$ and $\varepsilon\leq\min\{r_1/(4K_1),\kappa\sigma/2\}$. For $\|\lambda-\bar\lambda\|\leq\sigma$, \eqref{eq:smallness} gives $\|a(\lambda)\|+\|b(\lambda)\|\leq\varepsilon+2\ell\|\lambda-\bar\lambda\|\leq r_1/(2K_1)$, so that Step~2 applies. With $z_\lambda:=z^\circ(\lambda)+w_\lambda$, the cyclic reconstruction \eqref{eq:y_reconstruction} reaches $y_f$ if and only if
\[
G_T(\lambda):=\frac1T\int_0^Tv(z_\lambda(t),\lambda)\,dt=\nu_T, \quad \text{i.e.}\quad v^\circ(\lambda)=\nu_T-h(\lambda),
\]
with $h(\lambda):=G_T(\lambda)-v^\circ(\lambda)$, see \eqref{eq:lambda_equation}. $h(\cdot)$ is continuous by Step~2, and by \eqref{eq:app_w_exp},
\begin{equation}
\|h(\lambda)\|\leq\frac{L_v}T\int_0^T\|w_\lambda(t)\|\,dt\leq\frac{c_1}T\big(\varepsilon+2\ell\|\lambda-\bar\lambda\|\big),\qquad c_1:=\frac{4KK_1L_v}{\mu}.
\label{eq:app_h}
\end{equation}
Let $T\geq T_3:=c_1(1+4\ell/\kappa)$. Then, by \eqref{eq:cyclic.smallness}, $\|\nu_T-h(\lambda)-\bar\nu\|\leq\varepsilon+\|h(\lambda)\|\leq\kappa\sigma$ for $\|\lambda-\bar\lambda\|\leq\sigma$, so that, by Lemma~\ref{lem:trim_velocity}, $\Psi(\lambda):=(v^\circ)^{-1}\big(\nu_T-h(\lambda)\big)$ is a continuous map from $\bar B(\bar\lambda,\sigma)$ into itself. By Brouwer's fixed point theorem, it has a fixed point $\lambda_T$, i.e.\ $G_T(\lambda_T)=\nu_T$. Moreover, by Lemma~\ref{lem:trim_velocity}(a) and \eqref{eq:app_h},
\[
\kappa\|\lambda_T-\bar\lambda\|\leq\|v^\circ(\lambda_T)-\bar\nu\|\leq\|\nu_T-\bar\nu\|+\|h(\lambda_T)\|\overset{\eqref{eq:cyclic.smallness}}{\leq}\varepsilon+\frac{c_1}T\big(\varepsilon+2\ell\|\lambda_T-\bar\lambda\|\big),
\]
and, since $2c_1\ell/T\leq\kappa/2$, we get $\|\lambda_T-\bar\lambda\|\leq\frac2\kappa(1+c_1)\varepsilon$.

\emph{Step 4 (the extremal and its localization).} Let $(x_T,p_{x,T}):=z_{\lambda_T}$, $u_T:=u^\star(z_{\lambda_T},\lambda_T)$ and $y_T(t):=y_0+\int_0^tv(z_{\lambda_T}(s),\lambda_T)\,ds$. By Steps~2--3, $(x_T,p_{x,T})$ solves \eqref{eq:RBVP} with $\lambda=\lambda_T$ and $y_T(T)=y_f$, so that $(x_T,y_T,p_{x,T},\lambda_T,u_T)$ is a normal extremal of \eqref{eq:OCP} by Corollary~\ref{cor:extremals}. By \eqref{eq:app_w_lambda} and Step~3,
\[
\|z_{\lambda_T}(t)-\bar z\|\leq\|w_{\lambda_T}\|_\infty+\|z^\circ(\lambda_T)-\bar z\|\leq2K_1\big(\varepsilon+2\ell\|\lambda_T-\bar\lambda\|\big)+\ell\|\lambda_T-\bar\lambda\|\leq C_2\,\varepsilon,
\]
and $\|u_T(t)-\bar u\|\leq\ell_u\big(\|z_{\lambda_T}(t)-\bar z\|+\|\lambda_T-\bar\lambda\|\big)$. Since $\|x-\bar x\|+\|p_x-\bar p_x\|\leq\sqrt2\|z-\bar z\|$, this proves \eqref{eq:localization} with $C_0:=(1+\ell_u)\big(\sqrt2\,C_2+\frac2\kappa(1+c_1)\big)$.

\emph{Step 5 (cost inequality and local optimality).} Since $E_1\succ0$ and by continuity, there are $\rho_K>0$ and $c>0$ such that $\nabla^2_{(x,u)}H^\lambda(x,p,u)\succeq c\,I$ whenever $\|x-\bar x\|+\|u-\bar u\|+\|p-\bar p_x\|+\|\lambda-\bar\lambda\|\leq\rho_K$. Let $\mathcal K:=\{(x,u):\|x-\bar x\|+\|u-\bar u\|\leq\rho_K/2\}$, which is convex. Let $(x_1,y_1,p_1,\lambda_1,u_1)$ be a normal extremal of \eqref{eq:OCP} with $(x_1(t),u_1(t))\in\mathcal K$ and $\|p_1(t)-\bar p_x\|+\|\lambda_1-\bar\lambda\|\leq\rho_K/2$ for all $t$, and let $(x,y,u)$ be any admissible trajectory of \eqref{eq:OCP} with $(x(t),u(t))\in\mathcal K$ for a.e.\ $t$. We claim that
\begin{equation}
J_T(u,x)-J_T(u_1,x_1)\geq\frac c2\int_0^T\big(\|\delta x\|^2+\|\delta u\|^2\big)\,dt, ~~ \delta x:=x-x_1,\ \delta u:=u-u_1,
\label{eq:app_convexity}
\end{equation}
where $J_T$ is the cost of \eqref{eq:OCP}. Indeed, let $H_1(x,u):=H^{\lambda_1}(x,p_1(t),u)$, so that $f^0(x)+\frac12u^\top Ru=H_1(x,u)-p_1^\top\dot x-\lambda_1^\top\dot y$ along both trajectories. The terms in $\lambda_1$ cancel after integration, since both trajectories go from $y_0$ to $y_f$. Integrating by parts, with $\delta x(0)=\delta x(T)=0$ and $\dot p_1=-\nabla_xH_1(x_1,u_1)$, gives $-\int_0^Tp_1^\top\delta\dot x\,dt=-\int_0^T\nabla_xH_1(x_1,u_1)^\top\delta x\,dt$. Since moreover $\nabla_uH_1(x_1,u_1)=0$,
\[
J_T(u,x)-J_T(u_1,x_1)=\int_0^T\Big(H_1(x,u)-H_1(x_1,u_1)-\nabla_{(x,u)}H_1(x_1,u_1)\cdot(\delta x,\delta u)\Big)\,dt,
\]
and \eqref{eq:app_convexity} follows from Taylor's formula, since the segment between $(x_1(t),u_1(t))$ and $(x(t),u(t))$ lies in $\mathcal K$, where the Hessian of $H_1$ is bounded below by $c\,I$.

Now let $\varepsilon_0\leq\min\{r_1/(4K_1),\kappa\sigma/2\}$ be such that $C_0\varepsilon_0\leq\rho_K/4$. By \eqref{eq:localization}, the extremal of Step~4 satisfies the assumptions made on $(x_1,y_1,p_1,\lambda_1,u_1)$, and \eqref{eq:app_convexity} gives (iii). Since $(x_T(t),u_T(t))$ lies at distance at least $\rho_K/4$ from the boundary of $\mathcal K$, and since, for fixed $T$, $\|x-x_T\|_\infty$ is small when $\|u-u_T\|_\infty$ is small (Gronwall's lemma), $u_T$ is also a strict local minimizer for the $L^\infty$ topology on the control; the radius of this $L^\infty$-neighborhood may depend on $T$, while $\mathcal K$ does not.

\emph{Step 6 (local uniqueness).} Let $\rho>0$ be such that $\rho\leq\rho_\Lambda/2$, $(1+\ell)\rho\leq r$ and $(1+\ell_u)\rho\leq\rho_K/2$, and decrease $\varepsilon_0$ so that $C_0\varepsilon_0\leq\rho$. Let $(x,y,p_x,\lambda,u)$ be a normal extremal in the $\rho$-tube. Since $u=u^\star(x,p_x,\lambda)$, it satisfies the assumptions made on $(x_1,y_1,p_1,\lambda_1,u_1)$ in Step~5. Applying \eqref{eq:app_convexity} twice, with the roles of this extremal and of the extremal of Step~4 exchanged, and adding, we get $\int_0^T(\|\delta x\|^2+\|\delta u\|^2)\,dt\leq0$. Hence $(x,u)=(x_T,u_T)$ and $y=y_T$.

It remains to compare the multipliers. Let $\psi:=(p_x-p_{x,T},\lambda-\lambda_T)$, and let $\mathcal A(t)=\begin{pmatrix}D_x(f_1+F_2u_T)&0\\ D_x(g_1+G_2u_T)&0\end{pmatrix}$, $\mathcal B(t)=\begin{pmatrix}F_2\\ G_2\end{pmatrix}$, evaluated at $x_T(t)$, define the linearization of the control system of \eqref{eq:OCP} along $(x_T,u_T)$. Since both extremals have the same trajectory, subtracting their adjoint equations and their conditions $\nabla_uH=0$ gives
\[
\dot\psi=-\mathcal A(t)^\top\psi,\qquad \mathcal B(t)^\top\psi(t)=0,\qquad t\in[0,T].
\]
Let $U(t,s)$ be the transition matrix of $\dot\zeta=\mathcal A(t)\zeta$, so that $\psi(s)=U(0,s)^\top\psi(0)$, and let $\mathcal G:=\int_0^1U(0,s)\mathcal B(s)\mathcal B(s)^\top U(0,s)^\top ds$ (recall $T\geq T_2\geq1$). Then $\psi(0)^\top\mathcal G\,\psi(0)=\int_0^1\|\mathcal B(s)^\top\psi(s)\|^2ds=0$. For the constant pair $(\mathcal A,\mathcal B)$ of (H2), the corresponding matrix $\int_0^1e^{-\mathcal As}\mathcal B\mathcal B^\top e^{-\mathcal A^\top s}ds$ is positive definite by the Kalman condition. By \eqref{eq:localization}, $(\mathcal A(t),\mathcal B(t))$ is uniformly $O(\varepsilon)$-close to $(\mathcal A,\mathcal B)$, and $\mathcal G$ depends continuously on the coefficients. Hence, after decreasing $\varepsilon_0$, $\mathcal G$ is positive definite, which gives $\psi\equiv0$. An abnormal lift of $(x_T,y_T,u_T)$ would satisfy the same system, with $\psi=(p_x,\lambda)\neq0$, since $p^0=0$ removes the terms coming from the cost; it is excluded by the same argument. This proves (ii), and the proposition holds with these $\varepsilon_0$, $\rho$, $C_0$, $\mathcal K$ and $T_0:=\max\{T_2,T_3\}$. \hfill$\square$

\end{appendices}

\end{document}